\documentclass[english]{article}
\usepackage[T1]{fontenc}
\usepackage[latin9]{inputenc}
\usepackage{geometry}
\usepackage{color}
\usepackage{babel}
\usepackage{mathrsfs}
\usepackage{algorithm2e}
\usepackage{amsmath}
\usepackage{amsthm}
\usepackage{amssymb}
\usepackage{stmaryrd}
\ifx\hypersetup\undefined
  \AtBeginDocument{%
    \hypersetup{unicode=true,pdfusetitle,
 bookmarks=true,bookmarksnumbered=false,bookmarksopen=false,
 breaklinks=false,pdfborder={0 0 1},backref=false,colorlinks=true}
  }
\else
  \hypersetup{unicode=true,pdfusetitle,
 bookmarks=true,bookmarksnumbered=false,bookmarksopen=false,
 breaklinks=false,pdfborder={0 0 1},backref=false,colorlinks=true}
\fi

\makeatletter
\newcommand{\lyxaddress}[1]{
	\par {\raggedright #1
	\vspace{1.4em}
	\noindent\par}
}
\theoremstyle{plain}
\newtheorem{thm}{\protect\theoremname}[section]
\theoremstyle{definition}
\newtheorem{defn}[thm]{\protect\definitionname}
\theoremstyle{remark}
\newtheorem{rem}[thm]{\protect\remarkname}
\theoremstyle{plain}
\newtheorem{lem}[thm]{\protect\lemmaname}
\theoremstyle{plain}
\newtheorem{prop}[thm]{\protect\propositionname}
\theoremstyle{plain}
\newtheorem{cor}[thm]{\protect\corollaryname}
\theoremstyle{definition}
\newtheorem{example}[thm]{\protect\examplename}
\newenvironment{lyxlist}[1]
	{\begin{list}{}
		{\settowidth{\labelwidth}{#1}
		 \setlength{\leftmargin}{\labelwidth}
		 \addtolength{\leftmargin}{\labelsep}
		 }}
	{\end{list}}
\theoremstyle{plain}
\newtheorem*{prop*}{\protect\propositionname}
\theoremstyle{plain}
\newtheorem*{lem*}{\protect\lemmaname}

\usepackage{stmaryrd}
\usepackage{algorithm2e}

\usepackage[colorlinks]{hyperref}

\usepackage{mathabx}
\usepackage{pdfcomment}
\usepackage{verbatim}

\newcounter{hypA}
\newenvironment{hyp}{\refstepcounter{hypA}\begin{itemize}
\item[({\bf A\arabic{hypA}})]}{\end{itemize}}

\makeatother

\providecommand{\corollaryname}{Corollary}
\providecommand{\definitionname}{Definition}
\providecommand{\examplename}{Example}
\providecommand{\lemmaname}{Lemma}
\providecommand{\propositionname}{Proposition}
\providecommand{\remarkname}{Remark}
\providecommand{\theoremname}{Theorem}

\begin{document}
\title{Peskun-Tierney ordering for Markov chain and process Monte Carlo:
beyond the reversible scenario}
\author{Christophe Andrieu \& Samuel Livingstone}
\maketitle

\lyxaddress{School of Mathematics, University of Bristol, UK.}

\lyxaddress{Department of Statistical Science, University College London, UK.}
\begin{abstract}
Historically time-reversibility of the transitions or processes underpinning
Markov chain Monte Carlo methods (MCMC) has played a key r\^ole in
their development, while the self-adjointness of associated operators
together with the use of classical functional analysis techniques
on Hilbert spaces have led to powerful and practically successful
tools to characterize and compare their performance. Similar results
for algorithms relying on nonreversible Markov processes are scarce.
We show that for a type of nonreversible Monte Carlo Markov chains
and processes, of current or renewed interest in the Physics and Statistical
literatures, it is possible to develop comparison results which closely
mirror those available in the reversible scenario. We show that these
results shed light on earlier literature, proving some conjectures
and strengthening some earlier results.
\end{abstract}

\section{Introduction}

Markov chain Monte Carlo (MCMC) is concerned with the simulation of
realisations of $\pi-$invariant and ergodic Markov chains, where
$\pi$ is a probability distribution of interest defined on some appropriate
measurable space $(\mathsf{X},\mathscr{X})$. Such realisations can
be used to produce samples of distributions arbitrarily close to $\pi$,
or approximate expectations with respect to $\pi$. For a given probability
distribution $\pi$ the choice of a Markov chain is not unique, and
understanding the nature of the approximation associated to particular
choices is therefore of importance and has generated a substantial
body of literature, both in statistical science and physics--among
others and directly related to our work \cite{peskun1973optimum,caracciolo1990nonlocal,liu1996peskun,tierney1998,greenwood1998information,mira2001ordering,neal2004improving,leisen2008extension,hobert2008,Latuszynski2013,maire2014comparison,roberts2014minimising,rosenthal2015spectral,doucet2015efficient,sakai2016eigenvalue,rey2016improving,andrieu2016establishing,Bornn2017,10.1093/biomet/asx031}.
The present paper is a contribution to this literature and addresses
a scenario currently barely covered by existing theory, despite recent
interest motivated by applications.

Due to its wide applicability the Metropolis-Hastings update \cite{metropolis1953equation,hastings1970monte}
is the cornerstone of the design of general purpose MCMC algorithms.
The corresponding Markov transition probability satisfies the so-called
detailed balance property, ensuring that $\pi$ is left invariant
by this update, but also implies reversibility of the numerous algorithms
of which it is a building block. An unintended benefit of reversibility
is given at a theoretical level. Using the operator interpretation
of Markov transition probabilities, the properties of reversible Markov
chains can be studied with well established functional analysis techniques
developed for self-adjoint operators. The celebrated result of Peskun
and its extensions \cite{peskun1973optimum,caracciolo1990nonlocal,tierney1998}
are an example (see \cite{mira2001ordering} for a review), and allow
for practical performance comparisons in numerous scenarios of interest
(see Theorem \ref{thm:carraciolo-tierney} and its Corollary for a
quick reference), providing in particular clear answers to questions
concerned with the design of algorithms. While reversibility facilitates
theoretical analysis and has historically enabled methodological developments,
it is not necessarily a desirable property when performance is considered.
Informally such processes have a tendency to ``backtrack'', slowing
down exploration of the support of the target distribution $\pi$.

Recently, there has been renewed interest in the design of $\pi-$invariant
Markov chains which are not reversible. In several specific scenarios
it has been shown that departing from reversibility can both improve
the speed of convergence of a Markov chain \cite{diaconis2000analysis},
and reduce the asymptotic variance of resulting estimators \cite{sakai2016eigenvalue}
(although counterexamples also exist \cite{roberts2015surprising}).
Certain nonreversible samplers have been known for some time \cite{horowitz1991generalized,gustafson1998guided},
but interest has been re-kindled more recently thanks to a suite of
methods which are not instances of the Metropolis--Hastings class.
All these Markov transition probabilities share a common structure,
illustrated here with a very simple example. Assume that $\mathsf{X=\mathbb{Z}}$,
let $E:=\mathsf{X}\times\{-1,1\}$, embed the distribution of interest
$\pi$ into $\mu(x,v):=\frac{1}{2}\pi(x)\mathbb{I}\big\{ v\in\{-1,1\}\big\}$
and consider the Markov transition probability
\begin{equation}
P(x,v;y,v):=\alpha(x,v)\mathbb{I}\{y=x+v\}\mathbb{I}\{w=v\}+\mathbb{I}\{y=x\}\mathbb{I}\{w=-v\}\big[1-\alpha(x,v)\big],\label{eq:toy-gustafson}
\end{equation}
where $\alpha(x,v):=\min\big\{1,\pi(x+v)/\pi(x)\big\}$ and $\mathbb{I}S$
is the indicator function of set $S$. In words, starting at $(x,v)$,
the first component of the Markov chain generated by $P$ will travel
in the same direction $v$ in increments of size one until a rejection
occurs and the direction is reversed. One can check that this does
not satisfy detailed balance with respect to $\mu$ (or indeed $\pi$),
but a similar looking property
\[
\mu(x,v)P(x,v;y,w)=\mu(y,w)P(y,-w;x,-v),
\]
for $(x,v),(y,w)\in E$. This is referred to as modified detailed
balance in the literature \cite{fang2014compressible} or skewed detailed
balance \cite{hukushima2013irreversible}, and leads to what is known
as Yaglom reversibility \cite{yaglom1949statistical}. It is instructive
to write this identity in terms of the transition probability $Q(x,v;y,w):=\mathbb{I}\{y=x\}\mathbb{I}\{w=-v\}$,
which now reads $\mu(x,v)P(x,v;y,w)=\pi(y,w)QPQ(y,w;x,v)$ where $QPQ$
is the composition of the three kernels. An interpretation of this
identity is that the corresponding operator $QPQ$ is the adjoint
of $P$, not $P$ as is the case in the self-adjoint scenario. This
structure of the adjoint of $P$, together with the fact that $Q^{2}$
is the identity, play a central r\^ole in our analysis and covers
a surprisingly large number of known scenarios and applications currently
beyond the reach of earlier theory. Indeed our theory does not require
the embedding $\mu$ of $\pi$ to be of the specific form above, and
$Q$ is only required to be an isometric involution--see Subsection
\ref{subsec:Q-self-adjointness} for a precise definition in the present
context. As we shall see this structure allows us to develop a theory
for performance comparison for this class of MCMC algorithms which
parallels that existing for reversible algorithms; see Subsection
\ref{subsec:DT-ordering-results}. Applications are given in Section
\ref{subsec:DT-Examples}, and include the proof of conjectures concerned
with the lifted Metropolis--Hastings method of \cite{turitsyn2011irreversible,vucelja2016lifting}
and improves and generalises the results of \cite{sakai2016eigenvalue},
provide a direct and rigorous proof of \cite{neal2004improving} in
a more general set-up and a connection to the results of \cite{maire2014comparison},
which is generalised, permitting the characterization of algorithms
(e.g. \cite{horowitz1991generalized,campos2015extra}) currently not
covered by existing theory. 

We show that this structure is shared by nonreversible Markov process
Monte Carlo (MPMC) methods, the continuous time pendant of MCMC, which
have recently attracted some attention \cite{peters2012rejection,2015bou-rabee-sanz-serna,bouchard2015bouncy,bierkens2015piecewise,bierkens2016zig,ottobre2016markov}.
Characterization of this property in the continuous time setup is
precisely formulated in Subsection \ref{subsec:continuous:Set-up-and-characterization}
and a concrete example discussed in Subsection \ref{subsec:PDMP-mu,Q-symmetry}.
In Subsection \ref{subsec:continuous-Ordering-of-asymptotic} we propose
new tools which enable performance comparison for this class of processes
and an application is presented in Section \ref{sec:Continuous-time-scenario-applications}.

Throughout this paper we will use the following standard notation.
Let $\big(E,\mathscr{E}\big)$ be a measurable space. For Markov kernels
$T_{1},T_{2}\colon E\times\mathscr{E}\rightarrow[0,1]$ we let $T_{1}T_{2}(z,A):=\int T_{1}(z,{\rm d}z')T_{2}(z',A)$
for all $A\in\mathscr{E}$ and for any probability measure $\nu$
on $\big(E,\mathscr{E}\big)$ and $f\in\mathbb{R}^{E}$ measurable,
we let $\nu(f):=\int f{\rm d}\nu$, sometimes simplified to $\nu f$
when no ambiguity is possible and whenever this quantity exists. We
denote $T$ the associated operators acting on functions to the right
as $Tf(z):=\int f(z')T(z,{\rm d}z')$ for $z\in E$, and on measures
to the left as $\nu T(A):=\int_{E}\int_{A}\nu({\rm d}z)T(z,{\rm d}z')$
for every $A\in\mathscr{E}$. Let $\mu$ be a probability distribution
defined on some measurable space $\big(E,\mathscr{E}\big)$. Whenever
the following exist, for $f,g:E\rightarrow\mathbb{R}$, we define
$\bigl\langle f,g\bigr\rangle_{\mu}:=\int fg{\rm d}\mu$, $\|f\|_{\mu}:=\smash{\big(\int f^{2}{\rm d}\mu\big)^{1/2}}$
and the Hilbert spaces
\[
L^{2}(\mu):=\bigl\{ f\in\mathbb{R}^{E}:\|f\|_{\mu}<\infty\bigr\},
\]
and $L_{0}^{2}(\mu):=L^{2}(\mu)\cap\{f\in\mathbb{R}^{E}:\mu(f)=0\}$.
We let $\vvvert T\vvvert_{\mu}:=\sup_{\|f\|_{\mu}=1}\|Tf\|_{\mu}$
and denote $T^{*}$ the adjoint of $T$, whenever it is well defined.
For a set $S$ we let $S^{c}$ be its complement in the ambient space.

\section{\label{sec:Discrete-time-section}Discrete time scenario -- general
results}

\subsection{The notion of $(\mu,Q)-$self-adjointness\label{subsec:Q-self-adjointness}}

Here we formalize the notion of $(\mu,Q)-$self-adjointness, and discuss
its consequences in the discrete time setting.
\begin{defn}
\label{def:isometric-involution}We call a linear operator $Q\colon L^{2}(\mu)\to L^{2}(\mu)$
an \emph{isometric involution} if
\end{defn}

\begin{enumerate}
\item $\langle f,g\rangle_{\mu}=\langle Qf,Qg\rangle_{\mu}$ for all $f,g\in L^{2}(\mu)$,
\item $Q^{2}={\rm Id}$, the identity operator.
\end{enumerate}
\begin{rem}
\label{rem:propertiesQ}We note the simple properties
\begin{itemize}
\item $Q$ is $\mu-$self-adjoint since $\langle f,Qg\rangle_{\mu}=\langle Qf,Q^{2}g\rangle_{\mu}=\langle Qf,g\rangle_{\mu}$
for $f,g\in L^{2}(\mu)$,
\item the operators $\Pi_{+}:=({\rm Id}+Q)/2$ and $\Pi_{-}:=({\rm Id}-Q)/2$
are $\mu-$self-adjoint projectors and for any $f\in L^{2}(\mu)$,
$f=\Pi_{+}f+\Pi_{-}f$, $Q\Pi_{+}f=\Pi_{+}f$ and $Q\Pi_{-}f=-\Pi_{-}f$.
\item for $\Pi$ an orthogonal projector, $Q=\pm\big({\rm Id}-2\Pi\big)$
is an isometric involution.
\end{itemize}
\end{rem}

Again we will use the same symbol for the associated Markov kernel
$Q\colon E\times\mathscr{E}\rightarrow[0,1]$ and that $\mu Q=\mu$.
The following establishes that there exists an involution $\xi\colon E\rightarrow E$
such that for all $f\in\mathbb{R}^{E}$ and $z\in E$, $Qf(z)=f\circ\xi(z)$.
\begin{lem}
Let $T\colon E\times\mathscr{E}\rightarrow[0,1]$ be a Markov transition
probability such that for any $z\in E$, $T^{2}(z,\{z\})=1$, then
there exists an involution $\tau\colon E\rightarrow E$ such that
for $z,z'\in E$, $T(z,{\rm d}z')=\delta_{\tau(z)}({\rm d}z')$.
\end{lem}

\begin{proof}
Let $z\in E$ and $A_{z}:=\{z'\in E\colon T(z',\{z\})=1\}$. From
$T^{2}(z,\{z\})=\int T(z,{\rm d}z')T(z',\{z\})=1$ we deduce by contradiction
that $T(z,A_{z})=1$. Assume there exists $z_{1}',z_{2}'\in A_{z}$
such that $z_{1}'\neq z_{2}'$. Then for $i\in\{1,2\}$, $T^{2}(z_{i}',\{z_{i}'\})=\int T(z_{i}',{\rm d}z'')T(z'',\{z'_{i}\})=T(z,\{z'_{i}\})=1$,
which is not possible if $z_{1}'\neq z_{2}'$. Therefore $A_{z}$
is a singleton, whose element we denote $\tau(z)$ and $T(z,{\rm d}z')=\delta_{\tau(z)}({\rm d}z')$.
The involution property is immediate.
\end{proof}
\begin{defn}
We say a Markov operator $P\colon L^{2}(\mu)\to L^{2}(\mu)$ is $(\mu,Q)$-self-adjoint
if there is an isometric involution $Q$ such that for all $f,g\in L^{2}(\mu)$
\[
\langle Pf,g\rangle_{\mu}=\langle f,QPQg\rangle_{\mu}.
\]
We will say that the corresponding kernel $P\colon E\times\mathscr{E}\rightarrow[0,1]$
is $(\mu,Q)$-reversible. When $Q={\rm Id}$ we will simply say that
$P$ is $\mu-$self adjoint or $\mu-$reversible. The following is
a simple but important characterisation of $(\mu,Q)-$self-adjoint
operators. 
\end{defn}

\begin{prop}
\label{prop:QPsa}If the Markov operator $P$ is $(\mu,Q)-$self-adjoint
(resp. $\mu-$self-adjoint) then $QP$ and $PQ$ are $\mu-$self-adjoint
(resp. $(\mu,Q)-$self-adjoint). As a result a $(\mu,Q)-$self-adjoint
Markov operator is always the composition of two $\mu-$self-adjoint
Markov operators.
\end{prop}

\begin{proof}
Let $f,g\in L^{2}(\mu)$. Assume that $P$ is $\mu-$self-adjoint,
then $\langle QPf,g\rangle_{\mu}=\langle Pf,Qg\rangle_{\mu}=\langle f,PQg\rangle_{\mu}=\langle f,Q(QP)Qg\rangle_{\mu}$.
Similar arguments establish that $PQ$ is $(\mu,Q)-$self-adjoint.
Now assume that $P$ is $(\mu,Q)-$self-adjoint then $\langle QPf,g\rangle_{\mu}=\langle Pf,Qg\rangle_{\mu}=\langle f,QPQQg\rangle_{\mu}=\langle f,QPg\rangle_{\mu}$.
Similar arguments establish that $PQ$ is $\mu-$self-adjoint. The
last point is straightforward since $P=Q(QP)=(PQ)Q$.
\end{proof}
\begin{defn}
For $P$ a $(\mu,Q)-$self-adjoint operator we call $QP$ and $PQ$
its \emph{reversible parts}.
\end{defn}

\subsection{Ordering of asymptotic variances\label{subsec:DT-ordering-results}}

For an homogeneous Markov chain $\{Z_{0},Z_{1},\ldots\}$ of transition
kernel $P$ leaving $\mu$ invariant, started at equilibrium and any
$\mu$-measurable $f:E\to\mathbb{R}$, we define the \emph{asymptotic
variance}
\begin{equation}
{\rm var}(f,P):=\lim_{n\to\infty}n\text{{\rm var}}\Bigl(n^{-1}{\textstyle \sum_{i=0}^{n-1}}f(Z_{i})\Bigr),\label{eq:asymptotic-var-homogeneous}
\end{equation}
whenever the limit exists. This limit always exists, but may be infinite,
when $P$ is $\mu-$reversible and $f\in L^{2}(\mu)$. Beyond this
scenario general criteria exist \cite[Theorem 4.1]{maigret1978theoreme,glynn1996}
and often require a bespoke analysis. A general question of interest
is, given two Markov transition probabilities $P_{1}$ and $P_{2}$
leaving $\mu$ invariant, can one find a simple criterion to establish
that for some function $f$, ${\rm var}(f,P_{1})\geq{\rm var}(f,P_{2})$
or ${\rm var}(f,P_{1})\leq{\rm var}(f,P_{2})$, when these quantities
are well defined. When $P_{1}$ and $P_{2}$ are $\mu-$reversible
a criterion based on Dirichlet forms leads to a particularly simple
solution. Beyond this scenario little is known in general.

For $P$ and $f\in L^{2}(\mu)$, define the Dirichlet form
\begin{equation}
\mathcal{E}(f,P):=\langle f,({\rm Id}-P)f\rangle_{\mu}=\frac{1}{2}\int[f(z')-f(z)]^{2}\mu({\rm d}z)P(z,{\rm d}z').\label{eq:dirichlet-form}
\end{equation}
Note that this is $\|f\|_{\mu}^{2}-\langle f,Pf\rangle_{\mu}$ where
the last term is the first order auto-covariance coefficient for $f\in L_{0}^{2}(\mu)$.
Let ${\rm Gap}_{R}\big(P\big):=\inf_{f\in L_{0}^{2}(\mu),\|f\|_{\mu}\neq0}\mathcal{E}(f,P)/\|f\|_{\mu}^{2}$.
\begin{thm}[Caracciolo et al. \cite{caracciolo1990nonlocal}, Tierney \cite{tierney1998}]
\label{thm:carraciolo-tierney}Let $\mu$ be a probability distribution
on some measurable space $\big(E,\mathscr{E}\big)$, and let $P_{1}$
and $P_{2}$ be two $\mu-$reversible Markov transition probabilities.
If for any $g\in L^{2}(\mu)$, $\mathcal{E}(g,P_{1})\geq\mathcal{E}(g,P_{2})$,
then for any $f\in L^{2}(\mu)$ 
\[
\mathrm{var}(f,P_{1})\leq\mathrm{var}(f,P_{2})\text{ and }{\rm Gap}_{R}\bigl(P_{1}\bigr)\geq{\rm Gap}_{R}\bigl(P_{2}\bigr).
\]
\end{thm}

\begin{cor}[Peskun \cite{peskun1973optimum}]
\label{cor:peskun}Whenever for any $z\in E$ and $A\in\mathscr{E}$
such that $P_{1}(z,A\cap\{z\}^{c})\geq P_{2}(z,A\cap\{z\}^{c})$ then
for any $f\in L^{2}(\mu)$,
\[
\mathrm{var}(f,P_{1})\leq\mathrm{var}(f,P_{2})\text{ and }{\rm Gap}_{R}\bigl(P_{1}\bigr)\geq{\rm Gap}_{R}\bigl(P_{2}\bigr).
\]
\end{cor}

Our main result is that these conclusions extend in part to $(\mu,Q)-$reversible
transitions. In order to ensure the existence of the quantities we
consider, for any $\lambda\in[0,1)$ we introduce the $\lambda-$asymptotic
variance, defined for any $f\in L^{2}(\mu)$, with $\bar{f}:=f-\mu(f)$,
as
\[
{\rm var}_{\lambda}(f,P):=\|\bar{f}\|_{\mu}^{2}+2\sum_{k\geq1}\lambda^{k}\langle\bar{f},P^{k}\bar{f}\rangle_{\mu}=2\langle\bar{f},[{\rm Id}-\lambda P]^{-1}\bar{f}\rangle_{\mu}-\|\bar{f}\|_{\mu}^{2}.
\]
Whether $\lim_{\lambda\uparrow1}{\rm var}_{\lambda}(f,P)={\rm var}(f,P)$
when the latter exists is problem specific and not addressed here,
but we note that this is always true in the reversible scenario and
that a general sufficient condition is that $\sum_{k\geq1}\big|\langle\bar{f},P^{k}\bar{f}\rangle_{\mu}\big|<\infty$.
Rather we focus on ordering ${\rm var}_{\lambda}(P_{1},f)$ and ${\rm var}_{\lambda}(P_{2},f)$
for $\lambda\in[0,1)$ and leave the convergence to the asymptotic
variances as a separate problem. 
\begin{thm}
\label{thm:ordering_discrete}Let $\mu$ be a probability distribution
on some measurable space $\big(E,\mathscr{E}\big)$, and let $P_{1}$
and $P_{2}$ be two $(\mu,Q)-$reversible Markov transition probabilities.
Assume that for any $g\in L^{2}(\mu)$, $\mathcal{E}(g,QP_{1})\geq\mathcal{E}(g,QP_{2})$,
or for any $g\in L^{2}(\mu)$, $\mathcal{E}(g,P_{1}Q)\geq\mathcal{E}(g,P_{2}Q)$.
Then for any $\lambda\in[0,1)$ and $f\in L^{2}(\mu)$
\end{thm}

\begin{enumerate}
\item satisfying $Qf=f$ it holds that ${\rm var}_{\lambda}(f,P_{1})\leq{\rm var}_{\lambda}(f,P_{2})$,
\item satisfying $Qf=-f$ it holds that ${\rm var}_{\lambda}(f,P_{1})\geq{\rm var}_{\lambda}(f,P_{2}).$
\end{enumerate}
\begin{cor}
If $P_{1}$ and $P_{2}$ are such that for every $z\in E$ and every
$A\in\mathcal{\mathscr{E}}$ it holds that $P_{1}Q(z,A\cap\{z\}^{c})\geq P_{2}Q(z,A\cap\{z\}^{c})$,
or $QP_{1}(z,A\cap\{z\}^{c})\geq QP_{2}(z,A\cap\{z\}^{c})$, then
the conclusion of Theorem \ref{thm:ordering_discrete} holds.
\end{cor}

\begin{proof}[Proof of Theorem \ref{thm:ordering_discrete}]
 To compare $P_{1}$ and $P_{2}$ we follow the approach of \cite{tierney1998},
by introducing the mixture kernel $P(\beta)=\beta P_{1}+(1-\beta)P_{2}$
for $\beta\in[0,1]$ and establishing that the right derivative $\partial_{\beta}{\rm var}_{\lambda}\big(f,P(\beta)\big)$
is of constant sign for $\beta\in[0,1)$. Using the representation
${\rm var}_{\lambda}(f,P(\beta))=2\langle f,\big[{\rm Id}-\lambda P(\beta)\big]^{-1}f\rangle_{\mu}-\|f\|_{\mu}^{2}$
we have 
\begin{gather*}
\partial_{\beta}{\rm var}_{\lambda}\big(f,P(\beta)\big)=2\lambda\bigl\langle f,\big[{\rm Id}-\lambda P(\beta)\big]^{-1}\big(P_{1}-P_{2}\big)\big[{\rm Id}-\lambda P(\beta)\big]^{-1}f\bigr\rangle_{\mu},
\end{gather*}
which is justified in detail in \cite[Lemma 51]{andrieu2015convergence}.
Let $P$ be $(\mu,Q)-$self-adjoint, then for $f,g\in L^{2}(\mu)$,
\[
\bigl\langle f,\left[{\rm Id}-\lambda P\right]^{-1}g\bigr\rangle_{\mu}=\sum_{k=0}^{\infty}\lambda^{k}\langle(QPQ)^{k}f,g\rangle_{\mu}=\sum_{k=0}^{\infty}\lambda^{k}\langle QP^{k}Qf,g\rangle_{\mu}=\bigl\langle Q\left[{\rm Id}-\lambda P\right]^{-1}Qf,g\bigr\rangle_{\mu}.
\]
Noting that $P(\beta)$ is $(\mu,Q)-$self-adjoint, the above leads
to
\begin{equation}
\partial_{\beta}{\rm var}_{\lambda}\big(f,P(\beta)\big)=\lambda\bigl\langle Q[{\rm Id}-\lambda P(\beta)]^{-1}Qf,(P_{1}-P_{2})[{\rm Id}-\lambda P(\beta)]^{-1}f\bigr\rangle_{\mu}.\label{eq:diff-var-lambda}
\end{equation}
Set
\begin{equation}
g:=\big[{\rm Id}-\lambda P(\beta)\big]^{-1}f,\label{eq:beta-sol-Poisson-discrete}
\end{equation}
then from the assumptions $Qf=\pm f$ and on the Dirichlet forms,
we deduce that
\begin{align*}
\partial_{\beta}{\rm var}_{\lambda}\big(f,P(\beta)\big) & =\pm\lambda\bigl\langle Qg,\big(P_{1}-P_{2}\big)g\bigr\rangle{}_{\mu},\\
 & =\pm\lambda\left\langle g,\big({\rm Id}-QP_{2}\big)g\right\rangle _{\mu}\mp\lambda\bigl\langle g,\big({\rm Id}-QP_{1}\big)g\bigr\rangle_{\mu}\lesseqqgtr0.
\end{align*}
Noting that $\bigl\langle Qg,\big(P_{1}-P_{2}\big)g\bigr\rangle{}_{\mu}=\bigl\langle Qg,\big(P_{1}-P_{2}\big)QQg\bigr\rangle{}_{\mu}$
and that $Q\colon L^{2}(\mu)\rightarrow L^{2}(\mu)$ is a bijection
we conclude.
\end{proof}
\begin{rem}
We note that in contrast with the reversible scenario the result never
provides us with information about the speed of convergence to equilibrium.
The practical guideline resulting from the theorem is that after ``burn-in''
an algorithm should be tuned to maximize or minimize $\mathcal{E}(g,QP)$
or $\mathcal{E}(g,PQ)$ for all $g\in L^{2}(\mu)$.
\end{rem}

\begin{rem}
\label{rem:more-general-discrete-result-Poisson}From the proof it
can be seen that ordering ${\rm var}_{\lambda}(f,P_{1})$ and ${\rm var}_{\lambda}(f,P_{2})$
for a specific $f\in L^{2}(\mu)$ such that $Qf=f$, only requires
ordering Dirichlet forms for a particular subset of $L^{2}(\mu)$,
namely the $\lambda-$solutions of the Poisson equation (\ref{eq:beta-sol-Poisson-discrete}).
Although these quantities are generally intractable, in some scenarios
their structure may be exploited to order the Dirichlet forms involved.
Such ideas have been extensively used in the reversible scenario,
for example in \cite{andrieu2015convergence,andrieu2016establishing}
and we provide an example in the $(\mu,Q)-$self-adjoint setup in
Subsection \ref{subsec:equivalence}. Another consequence of this
is that strict inequalities can be obtained when the Dirichlet forms
are strictly ordered for nonconstant functions and these $\lambda-$solutions
are nonconstant. 
\end{rem}

\begin{rem}
More quantitative versions of this result, in the spirit of \cite{caracciolo1990nonlocal}
can also be replicated using ideas of \cite{Latuszynski2013}. For
$\alpha\in(0,1]$ consider the $(\mu,Q)-$reversible transition $P_{1}':=(1-\alpha){\rm Id}+\alpha P_{1}$,
notice that for $f\in L^{2}(\mu)$, $\langle\bar{f},[{\rm Id}-\lambda P_{1}']^{-1}\bar{f}\rangle_{\mu}=\alpha^{-1}\langle\bar{f},[{\rm Id}-\lambda P_{1}]^{-1}\bar{f}\rangle_{\mu}$
and for $g\in L^{2}(\mu)$
\begin{align*}
\langle g,[QP_{2}-QP_{1}']g\rangle_{\mu} & =\alpha\langle g,({\rm Id}-QP_{1})g\rangle_{\mu}-\langle g,({\rm Id}-QP_{2})g\rangle_{\mu}+(1-\alpha)\langle g,({\rm Id}-Q)g\rangle_{\mu}.
\end{align*}
Therefore if for all $g\in L^{2}(\mu)$, $\langle g,({\rm Id}-QP_{1})g\rangle_{\mu}\geq\alpha^{-1}\langle g,({\rm Id}-QP_{2})g\rangle_{\mu}$
then if in addition $f=Qf$,
\[
\alpha^{-1}\langle\bar{f},[{\rm Id}-\lambda P_{1}]^{-1}\bar{f}\rangle_{\mu}=\langle\bar{f},[{\rm Id}-\lambda P_{1}']^{-1}\bar{f}\rangle_{\mu}\leq\langle\bar{f},[{\rm Id}-\lambda P_{2}]^{-1}\bar{f}\rangle_{\mu},
\]
that is ${\rm var}_{\lambda}(f,P_{1})\leq(1-\alpha)\|\bar{f}\|_{\mu}^{2}+\alpha{\rm var}_{\lambda}(f,P_{2})$.
\end{rem}

\section{\label{subsec:DT-Examples}Discrete time scenario: examples}

The notion of $(\mu,Q)-$reversibility, often described in terms of
modified or skewed detailed balance, is known to hold for numerous
processes of interest but its implications, beyond establishing that
the corresponding Markov chain leaves $\mu$ invariant, are to the
best of our knowledge unknown. In this section we show that our framework
contributes to filling this gap and revisit a wide range of simple,
some foundational, questions. In some scenarios $(\mu,Q)-$reversibility
is not immediately apparent for a specific problem and we present
basic strategies to remedy this. More complex examples are possible,
such as extension of \cite{andrieu2015convergence,andrieu2016establishing}
or \cite{andrieu2016random}, for example, but beyond the scope of
this paper.

\subsection{\label{subsec:equivalence}Links to $2-$cycle based MCMC kernels}

Recently \cite{maire2014comparison}, have shown that results for
ordering of asymptotic variances of reversible time homogeneous Markov
chains can be extended to certain inhomogeneous Markov chains arising
naturally in the context of MCMC algorithms. Such chains are obtained
by cycling between two reversible MCMC kernels, and it is a natural
question to ask whether improving either of the kernels in terms of
individual Dirichlet forms improves performance of the inhomogeneous
chain resulting from their combination. Theorem \ref{thm:mdo-main-result}
below provides us with a simple and practical characterization. We
show that this result is in some sense dual to $(\mu,Q)-$reversibility
and provide a generalization which makes previously intractable analysis
of some algorithms possible. For $\pi$ a probability distribution
on some space $(\mathsf{X},\mathscr{X})$, $P_{1}$ and $P_{2}$ two
$\pi-$invariant Markov transitions and $f\in L^{2}(\pi)$, we extend
the definition of $\lambda-$asymptotic variance, for $\lambda\in[0,1)$
to the inhomogeneous scenario 
\[
{\rm var}_{\lambda}(f,\{P_{1},P_{2}\}):=\sum_{k\geq0}\lambda^{2k}\langle\bar{f},(P_{1}P_{2})^{k}({\rm Id}+\lambda P_{1})\bar{f}\rangle_{\pi}+\lambda^{2k}\langle\bar{f},(P_{2}P_{1})^{k}({\rm Id}+\lambda P_{2})\bar{f}\rangle_{\pi}-\|\bar{f}\|_{\pi}^{2},
\]
where $\bar{f}:=f-\pi(f)$, which is well defined since for any $g\in L^{2}(\pi)$
$\|P_{1}g\|_{\pi}\leq\|g\|_{\pi}$ and $\|P_{2}g\|_{\pi}\leq\|g\|_{\pi}$.
Under additional assumptions (see for example \cite[Proposition 9]{maire2014comparison})
the following limits exist and satisfy
\[
\lim_{\lambda\uparrow1}{\rm var}_{\lambda}(f,\{P_{1},P_{2}\})=\lim_{n\to\infty}n\text{{\rm var}}\Bigl(n^{-1}{\textstyle \sum_{i=0}^{n-1}}f(X_{i})\Bigr),
\]
where here $\{X_{0},X_{1},\ldots\}$ is the time inhomogeneous Markov
chain obtained by cycling through $P_{1}$ and $P_{2}$ and of initial
distribution $\pi$, that is for $A\in\mathscr{X},$ $\mathbb{P}\big(X_{k}\in A\mid X_{0},\ldots,X_{k-1}\big)=P_{2-(k\mod2)}(X_{k-1},A)$
for $k\geq1$ and $X_{0}\sim\pi$. The following is a reformulation
of \cite[Theorem 4 and Lemma 25]{maire2014comparison} combined with
a generalization of \cite[Lemma 18]{maire2014comparison}.
\begin{thm}[{see \cite[Theorem 4 and Lemma 25]{maire2014comparison}}]
\label{thm:mdo-main-result}Let $\pi$ be a probability distribution
defined on $(\mathsf{X},\mathscr{X})$. For $i,j\in\{1,2\}$, let
$P_{i,j}\colon\mathsf{X}\times\mathscr{X}\rightarrow[0,1]$ be $\pi-$reversible
Markov kernels such that for all $g\in L^{2}(\pi)$ and $i\in\{1,2\}$
we have $\mathcal{E}\big(g,P_{1,i}\big)\geq\mathcal{E}\big(g,P_{2,i}\big)$.
Then for any $f\in L^{2}(\pi)$ and $\lambda\in[0,1)$ 
\[
{\rm var}_{\lambda}(f,\{P_{1,1},P_{1,2}\})\leq{\rm var}_{\lambda}(f,\{P_{2,1},P_{2,2}\}).
\]
Further, if $f\in L^{2}(\pi)$ is such that $P_{i,1}f=f$ (or $P_{i,2}f=f$)
for $i\in\{1,2\}$, then
\[
{\rm var}_{\lambda}(f,P_{1,1}P_{1,2})\leq{\rm var}_{\lambda}(f,P_{2,1}P_{2,2}).
\]
\end{thm}

\begin{cor}
Let $Q$ be an isometric involution and $P_{1}$ and $P_{2}$ be $\pi-$reversible.
Note that for $i\in\{1,2\}$, $P_{i}=Q(QP_{i})$ (resp. $P_{i}=(P_{i}Q)Q$)
and that both $P_{i,1}:=Q$ (resp. $P_{i,1}:=P_{i}Q$) and $P_{i,2}:=QP_{i}$
(resp. $P_{i,2}:=Q$) are $\pi-$self-adjoint by Proposition \ref{prop:QPsa}.
We can therefore apply Theorem \ref{thm:mdo-main-result} and the
conclusion of Theorem \ref{thm:ordering_discrete} holds for $f\in L^{2}(\pi)$
such that $Qf=\pm f$. 
\end{cor}

Conversely one can show using a very simple argument that the first
statement of Theorem \ref{thm:mdo-main-result} is a direct consequence
of $(\mu,Q)-$reversibility of a particular time homogeneous chain,
where time is now part of the state, for a particular isometric involution.
Apart from linking two seemingly unrelated ideas, an interest of the
proof is that it highlights the difficulty with extending the results
to $m-$cycles with $m\geq3$.
\begin{proof}[Proof of Theorem \ref{thm:mdo-main-result}]
Here we let $E=\mathsf{X}\times\mathsf{V}$ with $\mathsf{V}:=\{1,2\}\times\{-1,1\}$,
let $v=(v_{1},v_{2})\in\mathsf{V}$ and consider the target distribution
$\mu({\rm d}(x,v))=\pi({\rm d}x)/4\mathbb{I}\{v\in\mathsf{V}\}$.
For $i\in\{1,2\}$ we define the Markov transition probabilities
\[
P_{i}(x,v;{\rm d}(y,w))=P_{i,v_{1}}(x,{\rm d}y)\mathbb{I}\{w_{1}=v_{1}\oplus v_{2},w_{2}=v_{2}\},
\]
where $1\oplus(\pm1)=2$ and $2\oplus(\pm1)=1$, and $Q(x,v;{\rm d}(y,w))=\delta_{x}({\rm d}y)\mathbb{I}\{w_{1}=v_{1}\oplus v_{2},w_{2}=-v_{2}\}$
whose corresponding operator is $\mu-$isometric and involutive. Notice
that for $f\in L^{2}(\mu)$ $i\in\{1,2\}$ and $(x,v)\in E$, 
\[
QP_{i}f(x,v)=P_{i}f\big(x,(v_{1}\oplus v_{2},-v_{2})\big)=P_{i,v_{1}\oplus v_{2}}f_{(v_{1},-v_{2})}(x),
\]
with $x\mapsto f_{v}(x):=f(x,v)$ and therefore
\begin{align}
\left\langle QP_{i}f,g\right\rangle _{\mu} & =\frac{1}{4}\sum_{v\in\mathsf{V}}\left\langle P_{i,v_{1}\oplus v_{2}}f_{(v_{1},-v_{2})},g_{(v_{1},v_{2})}\right\rangle _{\pi}=\frac{1}{4}\sum_{v\in\mathsf{V}}\left\langle f_{(v_{1},-v_{2})},P_{i,v_{1}\oplus v_{2}}g_{(v_{1},v_{2})}\right\rangle _{\pi}\label{eq:MDO_proof_mu_Q}\\
 & =\frac{1}{4}\sum_{v\in\mathsf{V}}\left\langle f_{(v_{1},v_{2})},P_{i,v_{1}\oplus v_{2}}g_{(v_{1},-v_{2})}\right\rangle _{\pi}=\left\langle f,QP_{i}g\right\rangle _{\mu},\nonumber 
\end{align}
where we have used that for $v\in\mathsf{V}$, $P_{i,v_{1}\oplus v_{2}}$
is $\pi-$self-adjoint and the property that $v_{1}\oplus(-v_{2})=v_{1}\oplus v_{2}$.
From Proposition \ref{prop:QPsa} $P_{i}$ is $(\mu,Q)-$self-adjoint
and we now follow the proof of Theorem \ref{thm:ordering_discrete}
and its notation. For $\beta\in[0,1)$, from (\ref{eq:diff-var-lambda})
and (\ref{eq:MDO_proof_mu_Q}) we deduce that for any $f\in L^{2}(\pi)$,
with $(x,v)\mapsto\breve{f}(x,v):=f(x)$ (which satisfies $Q\breve{f}=\breve{f}$),
\[
\partial_{\beta}{\rm var}_{\lambda}\big(\breve{f},P(\beta)\big)=\frac{1}{4}\sum_{v\in\mathsf{V}}\left\langle \big(P_{1,v_{1}\oplus v_{2}}-P_{2,v_{1}\oplus v_{2}}\big)g_{(v_{1},-v_{2})}(\beta),g_{(v_{1},v_{2})}(\beta)\right\rangle _{\pi},
\]
where
\[
g(\beta)(x,v)=[{\rm Id}-\lambda P(\beta)]^{-1}\breve{f}(x,v)=\sum_{k\geq0}\lambda^{2k}(P_{v_{1}}P_{v_{1}\oplus v_{2}})^{k}f(x_{1})+\lambda^{2k+1}(P_{v_{1}}P_{v_{1}\oplus v_{2}})^{k}P_{v_{1}}f(x).
\]
By noting that $g_{(v_{1},-v_{2})}(\beta)=g_{(v_{1},v_{2})}(\beta)$
since $v_{1}\oplus(-v_{2})=v_{1}\oplus v_{2}$ we deduce $\partial_{\beta}{\rm var}_{\lambda}\big(\breve{f},P(\beta)\big)\geq0$
and ${\rm var}_{\lambda}\big(\breve{f},P_{1}\big)\leq{\rm var}_{\lambda}\big(\breve{f},P_{2}\big)$.
The first claim follows from the fact that ${\rm var}_{\lambda}(\breve{f},P_{i})={\rm var}_{\lambda}(f,\{P_{i,1},P_{i,2}\})$
for $i\in\{1,2\}$. The second statement is immediate once we establish
that for $f\in L^{2}(\pi)$ such that $P_{i,1}f=f$ for $i\in\{1,2\}$,
then
\[
{\rm var}_{\lambda}(f,\{P_{i,1},P_{i,2}\})=\frac{2+\lambda+\lambda^{-1}}{2}{\rm var}_{\lambda^{2}}(f,P_{i,1}P_{i,2})+\frac{\lambda-\lambda^{-1}}{2}\|\bar{f}\|_{2}^{2}.
\]
Notice that
\begin{align*}
\sum_{k\geq0}\lambda^{2k}\langle\bar{f},(P_{i,2}P_{i,1})^{k}({\rm Id}+\lambda P_{i,2})\bar{f}\rangle_{\pi} & =\|\bar{f}\|_{\pi}^{2}+(1+\lambda^{-1})\sum_{k\geq1}\lambda^{2k}\langle\bar{f},(P_{i,2}P_{i,1})^{k}\bar{f}\rangle_{\pi},\\
\sum_{k\geq0}\lambda^{2k}\langle\bar{f},(P_{i,1}P_{i,2})^{k}({\rm Id}+\lambda P_{i,1})\bar{f}\rangle_{\pi} & =(1+\lambda)\sum_{k\geq0}\lambda^{2k}\langle\bar{f},(P_{i,1}P_{i,2})^{k}\bar{f}\rangle_{\pi},
\end{align*}
and from the definition of ${\rm var}_{\lambda}(f,\{P_{i,1},P_{i,2}\})$
and the fact that for $i\in\{1,2\}$ and $k\geq1$ $\langle\bar{f},(P_{i,2}P_{i,1})^{k}\bar{f}\rangle_{\pi}=\langle\bar{f},(P_{i,1}P_{i,2})^{k}\bar{f}\rangle_{\pi}$
we conclude. When $P_{i,2}f=f$ for $i\in\{1,2\}$ the result follows
from the case above and the symmetry ${\rm var}_{\lambda}(f,\{P_{i,2},P_{i,1}\})={\rm var}_{\lambda}(f,\{P_{i,1},P_{i,2}\})$.
\end{proof}
\begin{rem}
Note that the instrumental Markov chains introduced in the proof are
never ergodic, but can be marginally. It is possible to revisit this
proof for $m-$cycles and $m\geq3$, but the property $v_{1}\oplus(-v_{2})=v_{1}\oplus v_{2}$
fails in this scenario, in general, and it is not possible to conclude.
\end{rem}

Theorem \ref{thm:mdo-extension-main-result} below extends Theorem
\ref{thm:mdo-main-result} to 2-cycles of $(\mu,Q)-$reversible Markov
kernels--applications of this result are given in Subsection \ref{subsec:timereversible}.
\begin{thm}
\label{thm:mdo-extension-main-result}Let $\pi$ be a probability
distribution defined on some probability space $(\mathsf{X},\mathscr{X})$.
For $i,j\in\{1,2\}$, let $P_{i,j}\colon\mathsf{X}\times\mathscr{X}\rightarrow[0,1]$
be $(\mu,Q)-$reversible Markov kernels for some isometric involution
$Q$, and such that for all $i\in\{1,2\}$ we have $\mathcal{E}\big(g,QP_{1,i}\big)\geq\mathcal{E}\big(g,QP_{2,i}\big)$
for all $g\in L^{2}(\pi)$, or $\mathcal{E}\big(g,P_{1,i}Q\big)\geq\mathcal{E}\big(g,P_{2,i}Q\big)$
for all $g\in L^{2}(\pi)$. Then for any $f\in L^{2}(\pi)$ such that
$Qf=f$ and $\lambda\in[0,1)$
\[
{\rm var}_{\lambda}(f,\{P_{1,1},P_{1,2}\})\leq{\rm var}_{\lambda}(f,\{P_{2,1},P_{2,2}\}).
\]
Further, if $f\in L^{2}(\pi)$ is such that $P_{i,1}f=f$ (or $P_{i,2}f=f$)
for $i\in\{1,2\}$, then
\[
{\rm var}_{\lambda}(f,P_{1,1}P_{1,2})\leq{\rm var}_{\lambda}(f,P_{2,1}P_{2,2}).
\]
\end{thm}

\begin{proof}
Notice that for $f=Qf\in L^{2}(\pi)$ we have ${\rm var}_{\lambda}(f,\{P_{i,1},P_{i,2}\})={\rm var}_{\lambda}(f,\{P_{i,1}Q,QP_{i,2}\})$
since
\begin{align*}
\sum_{k\geq0}\lambda^{2k}\langle\bar{f},(P_{i,1}P_{i,2})^{k}({\rm Id}+\lambda P_{i,1})\bar{f}\rangle_{\pi} & =\sum_{k\geq0}\lambda^{2k}\langle\bar{f},(P_{i,1}QQP_{i,2})^{k}({\rm Id}+\lambda P_{i,1}Q)\bar{f}\rangle_{\pi},\\
\sum_{k\geq0}\lambda^{2k}\langle\bar{f},(P_{i,2}P_{i,1})^{k}({\rm Id}+\lambda P_{i,2})\bar{f}\rangle_{\pi} & =\sum_{k\geq0}\lambda^{2k}\langle\bar{f},(QP_{i,2}P_{i,1}Q)^{k}({\rm Id}+\lambda QP_{i,2})\bar{f}\rangle_{\pi},
\end{align*}
as for $k\geq0$, $(QP_{i,2}P_{i,1}Q)^{k}=Q(P_{i,2}P_{i,1})^{k}Q$,
$Q$ is an isometry, $Q^{2}=\mathrm{Id}$ and $Q\bar{f}=\bar{f}$.
From Proposition \ref{prop:QPsa} $P_{i,1}Q$ and $QP_{i,2}$ are
$\mu-$self-adjoint and we conclude with Theorem \ref{thm:mdo-main-result}.
Similarly one establishes that for $f=Qf\in L^{2}(\pi)$ we have ${\rm var}_{\lambda}(f,\{P_{i,1},P_{i,2}\})={\rm var}_{\lambda}(f,\{QP_{i,1},P_{i,2}Q\})$
and conclude in a similar way.
\end{proof}

\subsection{\label{subsec:timereversible}Construction of Markov kernels from
time-reversible flows}

A generic way to construct $(\mu,Q)-$reversible Markov transition
probability consists of the following slight generalization of \cite{horowitz1991generalized,fang2014compressible}.
For a probability distribution $m$ on $\big(E,\mathscr{E}\big)$
and measurable mapping $\psi\colon E\rightarrow E$ we let for any
$A\in\mathscr{E}$, $m^{\psi}(A):=m(\psi^{-1}(A))$. The presentation
parallels that of \cite[Section 2, second example.]{tierney1998}
in order to avoid specificities concerned with densities and, for
example, the presence of Jacobians.
\begin{prop}
\label{prop:metropolized-flows}Let $\mu$ be a probability distribution
on $\big(E,\mathscr{E}\big)$,
\begin{enumerate}
\item \label{enu:metropolized-flow-psi}$\psi\colon E\rightarrow E$ be
a bijection such that $\psi^{-1}=\xi\circ\psi\circ\xi$ for $\xi\colon E\rightarrow E$
corresponding to an isometric involution $Q$,
\item \label{enu:metropolized-flow-phi}$\phi\colon\mathbb{R}_{+}\rightarrow[0,1]$
such that $r\phi(r^{-1})=\phi(r)$ for $r>0$ and $\phi(0)=0$,
\item define for $z\in E$,
\[
r(z):=\begin{cases}
\frac{{\rm d}\mu^{\xi\circ\psi}/{\rm d}\nu(z)}{{\rm d}\mu/{\rm d}\nu(z)} & \text{if }{\rm d}\mu^{\xi\circ\psi}/{\rm d}\nu(z)>0\text{ and }{\rm d}\mu/{\rm d}\nu(z)>0,\\
0 & \text{otherwise}.
\end{cases}
\]
where $\nu:=\mu+\mu^{\xi\circ\psi}$
\end{enumerate}
then
\begin{equation}
P(z,{\rm d}z'):=\phi\circ r(z)\delta_{\psi(z)}({\rm d}z')+\delta_{\xi(z)}({\rm d}z')\big[1-\phi\circ r(z)\big],\label{eq:metropolize-flow}
\end{equation}
 is $(\mu,Q)-$reversible.

\end{prop}

\begin{proof}
This can be checked directly, but we instead check that $PQ$ is $\mu-$self-adjoint
and conclude with Proposition \ref{prop:QPsa}. For any measurable
and bounded $f\in\mathbb{R}^{E}$,
\begin{align*}
PQf(z) & =\phi\circ r(z)\,(Qf)\circ\psi(z)+\big[1-\phi\circ r(z)\big](Qf)\circ\xi(z)\\
 & =\phi\circ r(z)\,f\circ\xi\circ\psi(z)+\big[1-\phi\circ r(z)\big]f(z).
\end{align*}
The property on $\psi$ implies that $\xi\circ\psi\circ\xi\circ\psi={\rm Id}$,
that is $\xi\circ\psi$ is an involution and hence $PQ$ is $\mu-$reversible
from \cite[Section 2, second example.]{tierney1998}. From Proposition
\ref{prop:QPsa} $P$ is $(\mu,Q)-$self-adjoint.
\end{proof}
\begin{example}
\label{exa:time-reversal-hamilton}Assume $E=\mathsf{X}\times\mathsf{V}$
and for $(x,v)\in E$ and $f\in\mathbb{R}^{E}$ let $Qf(x,v):=f(x,-v)$.
Then for any $t\in\mathbb{R}$, $\psi_{t}(x,v)=(x+tv,v)$ satisfies
$\psi_{t}^{-1}=\xi\circ\psi_{t}\circ\xi$ and was considered in \cite{gustafson1998guided}
to define the Guided Random Walk Metropolis. More general examples
satisfying this condition include $\psi_{t}(x,v)=\psi_{t/2}^{B}\circ\psi_{t}^{A}\circ\psi_{t/2}^{B}(x,v)$
where $\psi_{t}^{A}(x,v):=(x+t\nabla_{v}H(x,v),v)$ and $\psi_{t}^{B}(x,v):=(x,v-t\nabla_{x}H(x,v))$
for a separable Hamiltonian $H\colon E\rightarrow\mathbb{R}$. This
is the Störmer--Verlet scheme considered in \cite{horowitz1991generalized}
to define the Hybrid Monte Carlo algorithm in the situation where
$H:=-\log{\rm d}\mu/{\rm d}\lambda^{{\rm Leb}}$ is well defined and
separable. More generally dynamical systems with the time reversal
symmetry (e.g. \cite{lamb1998time} and also \cite[Lemma 3.14]{faggionato2009non})
provide ways of constructing such mappings (see also \cite{poncet2017generalized,ottobre2016function,sohl2014hamiltonian,campos2015extra,ottobre2016markov}
and \cite{fang2014compressible}).
\end{example}

\begin{example}
Choices of $\phi(r)$ in Proposition \ref{prop:metropolized-flows}
include $\phi(r)=\min\{1,r\}$, which leads to the standard Metropolis-Hastings
acceptance probability, or $\phi(r)=r/(1+r)$ which corresponds to
Barker's dynamic. It is well known that for any $\phi$ satisfying
Proposition \ref{prop:metropolized-flows}-\ref{enu:metropolized-flow-phi}
one has $\phi(r)\leq\min\{1,r\}$ and that for Barker's choice $\frac{1}{2}\min\{1,r\}\leq\phi(r)$.
\end{example}

In order to be useful in practice a Markov transition of the type
given in (\ref{eq:metropolize-flow}) must be combined with another
transition in order to lead to an ergodic Markov chain \cite{gustafson1998guided,horowitz1991generalized}.
We focus here on $2-$cycles of $(\mu,Q)-$reversible Markov transitions.
\begin{thm}
\label{thm:metropolized-flow-2-cycle}Let $\mu$ be a probability
distribution on $\big(E,\mathscr{E}\big)$ and let $\psi$ satisfy
Proposition \ref{prop:metropolized-flows}-\ref{enu:metropolized-flow-psi}
for some isometric involution $Q$. Further for $i\in\{1,2\}$, let
$P_{i,2}$ be as in (\ref{eq:metropolize-flow}) for a mapping $\psi_{i}=\psi$
and some mapping $\phi_{i}$ satisfying Proposition \ref{prop:metropolized-flows}-\ref{enu:metropolized-flow-phi}
and let $P_{1,1}=P_{2,1}$ be a $(\mu,Q)-$reversible Markov transition.
Assume that $\phi_{1}\geq\phi_{2}$, then for any $f\in L^{2}(\mu)$
such that $Qf=f$ and $\lambda\in[0,1)$ we have
\[
{\rm var}_{\lambda}(f,\{P_{1,1},P_{1,2}\})\leq{\rm var}_{\lambda}(f,\{P_{2,1},P_{2,2}\}).
\]
In particular $\phi(r)=\min\{1,r\}$ achieves the smallest $\lambda-$asymptotic
variance.
\end{thm}

\begin{proof}
The expression in (\ref{eq:dirichlet-form}) leads to
\[
\mathcal{E}(g,P_{i,2}Q)=\frac{1}{2}\int\phi_{i}\circ r(z)[g\circ\xi\circ\psi(z)-g(z)]^{2}\mu({\rm d}z),
\]
for $i\in\{1,2\}$ and we conclude with Theorem \ref{thm:mdo-extension-main-result}.
\end{proof}
\begin{example}[Example \ref{exa:time-reversal-hamilton} (ctd)]
Assume here for presentational simplicity that $\mathsf{X}=\mathsf{V}=\mathbb{R}$
that $\mu$ has a density with respect to the Lebesgue measure and
$\mu(x,v)=\pi(x)\varpi(v)$ where $\varpi$ is a $\mathcal{N}(0,\sigma^{2})$
for some $\sigma^{2}>0$. In this setup a popular choice \cite{duane1987hybrid}
for $P_{1,1}=P_{2,1}$ is a momentum refreshment of the type, for
some $\omega\in(0,\pi/2]$,
\[
R_{\omega}\big((x,v);{\rm d}(y,w)\big)=\int\delta_{(x,v\cos\omega+v'\sin\omega)}\big({\rm d}(y,w)\big)\varpi({\rm d}v').
\]
Lemma \ref{lem:horowitz-update-Q-adjoint} below establishes that
the corresponding operator is $(\mu,Q)-$self-adjoint. We can therefore
apply Theorem \ref{thm:metropolized-flow-2-cycle} and deduce, for
example, that the choice $\phi(r)=\min\{1,r\}$ for all $\omega\in(0,\pi/2]$
is optimum. Further since $R_{\omega}(x,v;\{x\}\times\mathsf{V})=1$
we note that the second statement of Theorem \ref{thm:mdo-main-result}
holds, a result partially known for $\omega=\pi/2$ since in this
case for $i\in\{1,2\}$ $P_{i,1}P_{i,2}$ is $\mu-$reversible and
Theorem \ref{thm:carraciolo-tierney} can be applied.
\end{example}

\begin{lem}
\label{lem:horowitz-update-Q-adjoint}For any $\omega\in(0,\pi/2]$,
$R_{\omega}$ is $(\mu,Q)-$self-adjoint for $Q$ such that $Qf(x,v)=f(x,-v)$
for $f\in\mathbb{R}^{E}$.
\end{lem}

\begin{proof}
Let $f,g\in L^{2}(\mu)$. Note that for $x\in\mathsf{X}$,
\begin{align*}
\int f(x,v)QR_{\omega}Qg(x,v)\varpi({\rm d}v)=\int f(x,v)R_{\omega}Qg(x,-v)\varpi({\rm d}v) & =\int f(x,v)Qg(x,-v\cos\omega+w\sin\omega)\varpi({\rm d}v)\varpi({\rm d}w)\\
 & =\int f(x,v)g(x,v\cos\omega-w\sin\omega)\varpi({\rm d}v)\varpi({\rm d}w)\\
 & =\int f(x,v'\cos\omega+w'\sin\omega)g(x,v')\varpi({\rm d}v')\varpi({\rm d}w'),
\end{align*}
where on the last line we have used the change of variable $(v',w')=\left(v\cos\omega-w\sin\omega,v\sin\omega+w\cos\omega\right)$
and the fact that $\varpi\otimes\varpi$ is invariant by rotation.
We therefore deduce that $\bigl\langle f,QR_{\omega}Qg\bigr\rangle_{\mu}=\bigl\langle R_{\omega}f,g\bigr\rangle_{\mu}$
and conclude.
\end{proof}
Another application of the results above is the extra chance HMC method
presented in \cite{campos2015extra}, equivalent to the ideas of \cite{sohl2014hamiltonian},
which can be seen as an extension to Horowitz's scheme \cite{horowitz1991generalized}.
Using the notation of Proposition \ref{prop:metropolized-flows} the
main idea is to define a variation of (\ref{eq:metropolize-flow})
where transitions to $\xi\circ\psi(x,v),\xi\circ\psi\circ\psi(x,v),\ldots$
are attempted in sequence until success.
\begin{example}
Here $\mathsf{X}=\mathbb{R}$ for simplicity and $\mu\big({\rm d}(x,v)\big)=\pi({\rm d}x)\varpi({\rm d}v)$
where $\varpi({\rm d}v)$ is a $\mathcal{N}(0,\sigma^{2})$. With
$Qf(x,v)=f(x,-v)$ for $f\in\mathbb{R}^{E}$ and $\psi$ as in Proposition
\ref{prop:metropolized-flows}-\ref{enu:metropolized-flow-psi} we
let $\psi^{0}={\rm Id}$ and $\psi^{k}=\psi\circ\psi^{k-1}$ for $k\in\mathbb{N}\setminus\{0\}$.
Define for $K\in\mathbb{N}\setminus\{0\}$,
\[
P_{K}\big((x,v);{\rm d}(y,w)\big):=\sum_{k=1}^{K}\beta_{k}(x,v){\rm \delta}_{\psi^{k}(x,v)}\big({\rm d}(y,w)\big)+\rho_{K}(x,v){\rm \delta}_{\xi(x,v)}\big({\rm d}(y,w)\big),
\]
where, with $\alpha_{0}(x,v)=0$ and for $k=1,\ldots,K$ $\alpha_{k}(x,v)=\max\big\{\alpha_{k-1}(x,v),\min\{1,r_{k}(x,v)\}\big\}$,
with
\[
r_{k}(x,v):=\begin{cases}
\frac{{\rm d}\mu^{\xi\circ\psi^{k}}/{\rm d}\nu_{k}(z)}{{\rm d}\mu/{\rm d}\nu_{k}(z)} & \text{if }{\rm d}\mu^{\xi\circ\psi^{k}}/{\rm d}\nu_{k}(z)>0\text{ and }{\rm d}\mu/{\rm d}\nu_{k}(z)>0,\\
0 & \text{otherwise},
\end{cases}
\]
and $\nu_{k}:=\mu+\mu^{\xi\circ\psi^{k}}$, $\beta_{k}(x,v)=\alpha_{k}(x,v)-\alpha_{k-1}(x,v)$
and $\rho_{K}(x,v):=1-\sum_{k=1}^{K}\beta_{k}(x,v)$. It is shown
in \cite[Appendix A]{campos2015extra} that this update is $(\mu,Q)-$reversible,
while it is pointed out that for $\omega\in(0,\pi/2]$, $R_{\omega}P_{K}$
is not. We can apply Theorem \ref{thm:mdo-extension-main-result}
to deduce that for any $f\in L^{2}(\mu)$ such that $Qf=f$ and any
$\omega\in(0,\pi/2]$, the mapping $K\mapsto{\rm var}_{\lambda}\big(f,\{R_{\omega},P_{K}\}\big)$
is non increasing, since from Lemma \ref{lem:monotone-dirichlet-extra}
below, $K\mapsto\mathcal{E}\big(g,P_{K}Q\big)$ is nondecreasing.
In fact, since $R_{\omega}(x,v;\{x\}\times\mathsf{V})=1$, for $f\in L^{2}(\pi)$
and $\breve{f}(x,v):=f(x)$ for $(x,v)\in E$, we also deduce that
$K\mapsto{\rm var}_{\lambda}\big(\breve{f},R_{\omega}P_{K}\big)$
is nonincreasing.
\end{example}

\begin{lem}
\label{lem:monotone-dirichlet-extra}For any $g\in L^{2}(\mu)$, $K\mapsto\mathcal{E}\big(g,P_{K}Q\big)$
is non-decreasing.
\end{lem}

\begin{proof}
For $g\in L^{2}(\mu)$, we have from (\ref{thm:mdo-main-result})
\begin{align*}
\mathcal{E}\big(g,P_{K}Q\big) & =\frac{1}{2}\int\mu\big({\rm d}(x,v)\big)P_{K}Q\big((x,v);{\rm d}(y,w)\big)\big[g(x,v)-g(y,w)\big]^{2}\\
 & =\sum_{k=1}^{K}\frac{1}{2}\int\mu\big({\rm d}(x,v)\big)\beta_{k}(x,v)\big[g(x,v)-g\circ\xi\circ\psi^{k}(x,v)\big]^{2}.
\end{align*}
The result follows.
\end{proof}
\begin{rem}
As pointed out by \cite{campos2015extra} the rational behind the
approach is that for $(x,v)\in E$, $k\mapsto H\circ\xi\circ\psi^{k}(x,v)$
typically fluctuates around $H(x,v)$. As a result if there exist
$(x_{0},v_{0})\in E$ and $k_{0}\in\mathbb{N}_{*}$ such that $\min_{1\leq k\leq k_{0}}H\circ\xi\circ\psi^{k}(x_{0},v_{0})>\max\big\{ H(x_{0},v_{0}),H\circ\xi\circ\psi^{k_{0}+1}(x_{0},v_{0})\big\}$
and, for example, $(x,v)\mapsto H\circ\xi\circ\psi^{k_{0}+1}(x,v)$
is continuous in a neighbourhood of $(x_{0},v_{0})$ then $\mu\big(\{\beta_{k_{0}+1}(X,V)>0\}\big)>0$
and $\mathcal{E}\big(g,P_{k_{0}+1}Q\big)-\mathcal{E}\big(g,P_{k_{0}}Q\big)>0$
for $L^{2}(\mu)\ni g\neq g\circ\xi\circ\psi^{k_{0}+1}$ on the aforementioned
neighbourhood, suggesting that the strict performance improvement
observed numerically in \cite{campos2015extra} for specific functions
holds more generally. A more precise investigation of this point is
far beyond the scope of the present work.
\end{rem}

It is natural to try to assess the impact of $\omega\in(0,\pi/2]$
involved in the definition of $R_{\omega}$ on the performance of
the type of algorithms presented in this section. In particular, a
long-standing question is whether partial momentum refreshment is
preferable to full refreshment, meaning replacing $R_{\omega}$ by
$R_{\pi/2}$. Application of Theorem \ref{thm:mdo-extension-main-result}
requires establishing that $\langle g,Q(R_{\pi/2}-R_{\omega})g\rangle_{\mu}$
does not change sign for all $g\in L^{2}(\mu)$. This, however, is
not the case. For example, setting $g_{1}(x,v):=v$ then the quantity
is positive but for $g_{2}(x,v):=v^{2}$ it is negative and we cannot
conclude.

\subsection{Lifted MCMC algorithms}

Assume we are interested in sampling from $\pi$ defined on $(\mathsf{X,\mathscr{X}})$
and are given two sub-stochastic kernels $T_{1}$ and $T_{-1}$ such
that for $x,y\in\mathsf{X}$ the following ``skewed'' detailed balance
holds,
\begin{equation}
\pi({\rm d}x)T_{1}(x,{\rm d}y)=\pi({\rm d}y)T_{-1}(y,{\rm d}x).\label{eq:TCV-starting-point-MDB}
\end{equation}
A generic example, related to the Metropolis-Hastings algorithm is
as follows.
\begin{example}
Let $\big\{ q_{1}(x,\cdot),x\in\mathsf{X}\big\}$ and $\big\{ q_{-1}(x,\cdot),x\in\mathsf{X}\big\}$
be two families of probability distributions on $(\mathsf{X,\mathscr{X}})$,
then the kernel defined for $v\in\{-1,1\}$ and $x,y\in\mathsf{X}$
as
\[
T_{v}(x,{\rm d}y)=\min\left\{ 1,r_{v}(x,y)\right\} q_{v}(x,{\rm d}y)\text{ with }r_{v}(x,y):=\begin{cases}
\frac{{\rm d}\gamma_{-v}/{\rm d}\nu(y,x)}{{\rm d}\gamma_{v}/{\rm d}\nu(x,y)} & \text{if }{\rm d}\gamma_{v}/{\rm d}\nu(x,y)\times{\rm d}\gamma_{-v}/{\rm d}\nu(y,x)>0,\\
0 & \text{otherwise},
\end{cases}
\]
where $\gamma_{v}\big({\rm d}(x,y)\big):=\pi({\rm d}x)q_{v}(x,{\rm d}y)$
and $\nu\big({\rm d}(x,y)\big):=\gamma_{v}\big({\rm d}(x,y)\big)+\gamma_{-v}\big({\rm d}(y,x)\big)$.
\end{example}

A standard way of constructing a $\pi-$reversible Markov transition
probability based on the above sub-kernels consists of the following
\begin{equation}
P(x,{\rm d}y)=\frac{1}{2}T_{1}(x,{\rm d}y)+\frac{1}{2}T_{-1}(x,{\rm d}y)+\delta_{x}({\rm d}y)\left(1-\frac{1}{2}T_{1}(x,\mathsf{X})-\frac{1}{2}T_{-1}(x,\mathsf{X})\right).\label{eq:TCV-thereversiblechain}
\end{equation}
The standard Metropolis-Hastings algorithm corresponds to the scenario
where $T_{1}=T_{-1}$. The aim of the lifting strategy is to stratify
the choice between $T_{1}$ and $T_{-1}$ by embedding the sampling
problem into that of sampling from $\mu({\rm d}(x,v))=\pi({\rm d}x)\varpi(v)=\frac{1}{2}\pi({\rm d}x)\mathbb{I}\big\{ v\in\{-1,1\}\big\}$
and using a Markov kernel defined on the corresponding extended space
$E=\mathsf{X}\times\{-1,1\}$ which promotes contiguous uses of $T_{1}$
or $T_{-1}$ along the iterations. As shown in \cite{turitsyn2011irreversible,vucelja2016lifting}
one possible solution, which imposes $P^{{\rm lifted}}\big((x,v);(A\setminus\{x\})\times\{-v\}\big)=0$
for any $A\in\mathscr{X}$, is
\begin{align*}
P^{{\rm lifted}}\big((x,v);{\rm d}(y,w)\big)= & \mathbb{I}\{w=v\}\Bigl[T_{v}(x,{\rm d}y)+\delta_{x}({\rm d}y)\big(1-T_{v}(x,\mathsf{X})-\rho_{v,-v}(x)\big)\Bigr]+\mathbb{I}\{w=-v\}\delta_{x}({\rm d}y)\rho_{v,-v}(x),
\end{align*}
where $\rho_{1,-1}(x)$ and $\rho_{-1,1}(x)$ are free parameters,
the ``switching rates'', required to satisfy for all $(x,v)\in E$
$0\leq\rho_{v,-v}(x)\leq1-T_{v}(x,\mathsf{X})$ and
\begin{equation}
\rho_{v,-v}(x)-\rho_{-v,v}(x)=T_{-v}(x,\mathsf{X})-T_{v}(x,\mathsf{X}).\label{eq:TCV-condition-rho}
\end{equation}
It is not difficult to check that under (\ref{eq:TCV-starting-point-MDB})
and (\ref{eq:TCV-condition-rho}) $P^{{\rm lifted}}$ is $(\mu,Q)-$self-adjoint,
for $Q$ such that $Qf(x,v)=f(x,-v)$ for $f\in\mathbb{R}^{E}$. There
are numerous known solutions to the condition above \cite{hukushima2013irreversible},
including 
\[
\tilde{\rho}_{v,-v}(x):=\max\left\{ 0,T_{-v}(x,\mathsf{X})-T_{v}(x,\mathsf{X})\right\} .
\]
It is remarked as intuitive in \cite{vucelja2016lifting} that among
the possible solutions to (\ref{eq:TCV-condition-rho}), this choice
should promote fastest exploration. We prove below that this is indeed
true, in the sense that this choice minimizes asymptotic variances,
as a consequence of Theorem \ref{thm:ordering_discrete}. We let $P^{{\rm lifted},\rho}$
denote the transition probability which uses $\rho_{v,-v}$.
\begin{thm}
For any switching rate $\rho_{v,-v}$ satisfying $0\leq\rho_{v,-v}(x)\leq1-T_{v}(x,\mathsf{X})$
for all $(x,v)\in E$ and (\ref{eq:TCV-condition-rho}), any $f\in L^{2}(\mu)$
such that $Qf=f$ and $\lambda\in[0,1)$, we have
\[
{\rm var}_{\lambda}\big(f,P^{{\rm lifted},\tilde{\rho}}\big)\leq{\rm var}_{\lambda}\big(f,P^{{\rm lifted},\rho}\big)\leq{\rm var}_{\lambda}\big(f,P^{{\rm lifted},1-T_{v}}\big).
\]
\end{thm}

\begin{proof}
Let $\rho_{1,v,-v}$ and $\rho_{2,v,-v}$ be switching rates such
that $0\leq\rho_{1,v,-v}(x,v)\leq\rho_{2,v,-v}(x,v)\leq1-T_{v}(x,\mathsf{X})$
for all $(x,v)\in E$, then from the identity in (\ref{eq:dirichlet-form})
\begin{align*}
\mathcal{E}\big(g,P^{{\rm lifted},\rho_{1}}Q\big) & -\mathcal{E}\big(g,P^{{\rm lifted},\rho_{2}}Q\big)=\frac{1}{2}\int\mu\big({\rm d}(x,v)\big)\big(\rho_{2,v,-v}(x)-\rho_{1,v,-v}(x)\big)\Big[g(x,v)-g(x,-v)\Big]^{2}\geq0,
\end{align*}
and the application of Theorem \ref{thm:ordering_discrete} leads
to ${\rm var}_{\lambda}\big(f,P^{{\rm lifted},\rho_{1}}\big)\leq{\rm var}_{\lambda}\big(f,P^{{\rm lifted},\rho_{2}}\big)$
for any $f\in L^{2}(\mu)$ such that $Qf=f$. We now establish that
$\rho_{v,-v}$ satisfying $0\leq\rho_{v,-v}(x)\leq1-T_{v}(x,\mathsf{X})$
and (\ref{eq:TCV-condition-rho}) implies $\tilde{\rho}_{v,-v}(x,v)\leq\rho_{v,-v}(x,v)\leq1-T_{v}(x,\mathsf{X})$
for all $(x,v)\in E$, notice that $1-T_{v}(x,\mathsf{X})$ satisfies
(\ref{eq:TCV-condition-rho}) and apply the result above twice to
conclude. We proceed by contradiction to establish the first inequality.
Assume there exists a switching rate $\rho_{v,-v}$ such that for
some $(x,v)\in E$ such that $\tilde{\rho}_{v,-v}(x)>0$ we have $\rho_{v,-v}(x)<\tilde{\rho}_{v,-v}(x)$.
Then from (\ref{eq:TCV-condition-rho}), 
\[
\tilde{\rho}_{v,-v}(x)-\tilde{\rho}_{-v,v}(x)=\rho_{v,-v}(x)-\rho_{-v,v}(x),
\]
or equivalently
\[
\tilde{\rho}_{v,-v}(x)-\rho_{v,-v}(x)=\tilde{\rho}_{-v,v}(x)-\rho_{-v,v}(x)>0,
\]
which is impossible since $\tilde{\rho}_{v,-v}(x)>0$ implies $\tilde{\rho}_{-v,v}(x)=0$
and we must have $\rho_{-v,v}(x)\geq0$. Therefore we must necessarily
have $\rho_{v,-v}(x)\geq\tilde{\rho}_{v,-v}(x)$ for all $(x,v)\in E$.
\end{proof}
\begin{rem}
Readers familiar with the delayed rejection Metropolis-Hastings update
may notice the similarity here since
\begin{multline*}
P^{{\rm lifted}}\big((v,x);{\rm d}(w,y)\big)=\mathbb{I}\{w=v\}T_{v}(x,\mathsf{X})\frac{T_{v}(x,{\rm d}y)}{T_{v}(x,\mathsf{X})}\\
+\big[1-T_{v}(x,\mathsf{X})\big]\left[\mathbb{I}\{w=v\}\delta_{x}({\rm d}y)\left(1-\frac{\rho_{v,-v}(x)}{1-T_{v}(x,\mathsf{X})}\right)+\mathbb{I}\{w=-v\}\delta_{x}({\rm d}y)\frac{\rho_{v,-v}(x)}{1-T_{v}(x,\mathsf{X})}\right],
\end{multline*}
where we require the property
\[
\big[1-T_{v}(x,\mathsf{X})\big]\left(1-\frac{\rho_{v,-v}(x)}{1-T_{v}(x,\mathsf{X})}\right)=\big[1-T_{-v}(x,\mathsf{X})\big]\left(1-\frac{\rho_{-v,v}(x)}{1-T_{-v}(x,\mathsf{X})}\right),
\]
and notice that
\[
1-\frac{\tilde{\rho}_{v,-v}(x)}{1-T_{v}(x,\mathsf{X})}=\min\left\{ 1,\frac{1-T_{-v}(x,\mathsf{X})}{1-T_{v}(x,\mathsf{X})}\right\} .
\]
The theorem above establishes that this latter form of acceptance
probability for the second stage of the update is again optimum in
this setup. The update however differs from the standard delayed rejection
update in that here the accept/rejection probability is integrated,
restricting implementability of the approach. We also note that our
results can be used to established superiority of the standard delayed
rejection strategy in the context of $(\mu,Q)-$reversible updates
and that integration of the rejection probability in the scenario
above is beneficial.
\end{rem}

One can compare the performance of algorithms relying on $P^{{\rm lifted}}$
and $P$. With a slight abuse of notation for any $\lambda\in[0,1)$
and $f\in L^{2}(\pi)$ we let ${\rm var}_{\lambda}\big(f,P^{{\rm lifted}}\big)={\rm var}_{\lambda}\big(\breve{f},P^{{\rm lifted}}\big)$
where for $(x,v)\in E$ we let $\breve{f}(x,v):=f(x)$.
\begin{thm}
\label{prop:TCV-compare-rev-and-lifted}For any $\lambda\in[0,1)$
and $f\in L^{2}(\pi)$, any switching rate $\rho_{v,-v}$ satisfying
$\tilde{\rho}_{v,-v}(x,v)\leq\rho_{v,-v}(x,v)\leq1-T_{v}(x,\mathsf{X})$
for all $(x,v)\in E$, ${\rm var}_{\lambda}\big(f,P^{{\rm lifted,\rho}}\big)\leq{\rm var}_{\lambda}\big(f,P\big),$
with $P$ given in (\ref{eq:TCV-thereversiblechain}).
\end{thm}

\begin{proof}
Fix $\lambda\in[0,1)$. First consider the additive symmetrization
of $P^{{\rm lifted}}$, $S\big(P^{{\rm lifted}}\big):=\big[P^{{\rm lifted}}+\big(P^{{\rm lifted}}\big)^{*}\big]/2$,
which is $\mu-$self-adjoint. From a classical result, see for example
\cite[Lemma 2]{andrieu2016random}, we have for $g\in L^{2}(\mu)$,
\[
{\rm var}_{\lambda}\big(g,P^{{\rm lifted}}\big)\leq{\rm var}_{\lambda}\big(g,S\big(P^{{\rm lifted}}\big)\big).
\]
For $h\in\mathbb{R}^{E}$ measurable and bounded,
\begin{align*}
P^{{\rm lifted}}Qh(x,v) & =\int T_{v}(x,{\rm d}y)h(y,-v)+\big(1-T_{v}(x,\mathsf{X})-\rho_{v,-v}(x)\big)h(x,-v)+\rho_{v,-v}(x)h(x,v),
\end{align*}
and therefore
\[
QP^{{\rm lifted}}Qh(x,v)=\int T_{-v}(x,{\rm d}y)h(y,v)+\big(1-T_{-v}(x,\mathsf{X})-\rho_{-v,v}(x)\big)h(x,v)+\rho_{-v,v}(x)h(x,-v).
\]
Consequently, 
\begin{multline*}
S\big(P^{{\rm lifted}}\big)h(x,v)=\int\frac{1}{2}\big[T_{v}(x,{\rm d}y)+T_{-v}(x,{\rm d}y)\big]h(y,v)+\left(1-\frac{T_{v}(x,\mathsf{X})+T_{-v}(x,\mathsf{X})}{2}\right)h(x,v)\\
+\frac{\rho_{v,-v}(x)+\rho_{-v,v}(x)}{2}\big(h(x,-v)-h(x,v)\big).
\end{multline*}
Therefore, for $f\in L^{2}(\pi)$ and $(x,v)\in E$, $S(P^{{\rm lifted}})\breve{f}(x,v)=Pf(x)$
and by a straightforward induction one can establish $S(P^{{\rm lifted}})^{k}\breve{f}(x,v)=P^{k}f(x)$
for $k\geq1$.  As a consequence for $f\in L^{2}(\pi)$ and $\lambda\in[0,1)$
we have ${\rm var}_{\lambda}\big(f,P\big)={\rm var}_{\lambda}\big(\breve{f},S\big(P^{{\rm lifted}}\big)\big)$
and we conclude.
\end{proof}
\begin{example}
In the scenario where $\mathsf{X}=\mathbb{R}$ or $\mathsf{X}=\mathbb{Z}$
and $\pi$ has a density with respect to the Lebesgue or counting
measure, \cite{gustafson1998guided} introduced the guided walk Metropolis,
whose transition probability is
\[
P^{{\rm GRW}}\big((v,x);{\rm d}(w,y)\big)=T_{v}^{{\rm guided}}\big(x,{\rm d}y\big)\mathbb{I}\{w=v\}+\delta_{x}({\rm d}y)\mathbb{I}\{w=-v\}\big[1-T_{v}^{{\rm guided}}\big(x,\mathsf{X}\big)\big]
\]
where, 
\[
T_{v}^{{\rm guided}}(x,{\rm d}y):=\int_{\mathsf{X}}\min\left\{ 1,\frac{\pi(x+|z|v)}{\pi(x)}\right\} q({\rm d}z){\rm \delta}_{x+|z|v}({\rm d}y),
\]
for some symmetric distribution $q(\cdot)$ on $\mathsf{V}=\mathbb{R}$
or $\mathsf{V}=\{-1,1\}$. It is straightforward to check that $T_{v}^{{\rm guided}}$
satisfies (\ref{eq:TCV-starting-point-MDB}), and hence we can construct
a lifted version of Gustafson's algorithm. We also notice that $P$
corresponds in this case to the random walk Metropolis algorithm with
proposal distribution $q\big(\cdot\big)$--we denote this algorithm
$P^{{\rm RW}}$. Our two earlier results establish that for any switching
rate $\rho_{v,-v}$, $f\in L^{2}(\pi)$ and $\lambda\in[0,1)$,
\[
{\rm var}_{\lambda}\big(f,P^{{\rm lifted-GRW,\rho}}\big)\leq{\rm var}_{\lambda}\big(f,P^{{\rm GRW}}\big)\leq{\rm var}_{\lambda}\big(f,P^{{\rm RW}}\big).
\]
\end{example}

\subsection{Neal's scheme to avoid backtracking}

In \cite{neal2004improving} the author describes a generic way of
modifying a reversible Markov chain defined on a finite state space
$\mathsf{X}$ to reduce ``backtracking'' (a special case is also
discussed in \cite{diaconis2000analysis}). More specifically, assume
we are interested in sampling from some probability distribution $\pi$
defined on $\mathsf{X}$ and that we do so by using a $\pi$-reversible
(first order) Markov transition $T_{2}$ defined on $\mathsf{X}$.
Informally the idea in \cite{neal2004improving} is to modify the
first order Markov chain of transition $T_{2}$ into a second order
Markov chain $T_{1}$ to ensure that given a realization $X_{0},X_{1},\ldots,X_{k-1},X_{k}$
for some $k\geq1$ the new chain samples $X_{k+1}$ conditional upon
$X_{k}$ and $X_{k-1}$ and prevents the occurrence of the event $X_{k+1}=X_{k-1}$.
A probabilistic argument is developed in \cite{neal2004improving}
for $\mathsf{X}$ finite to establish that the resulting chain produces
estimators with an asymptotic variance that cannot exceed that of
estimators from the original chain. We show here that this holds more
generally for countable spaces and is a direct consequence of $(\mu,Q)-$self-adjointness
for a particular $Q$, the bivariate first order representation of
a second order univariate Markov chain, as used in \cite{neal2004improving}
and the application of Theorem \ref{thm:ordering_discrete}. For simplicity
of exposition we assume $0<T_{2}(x_{1},x_{2})<1$ for $x_{1},x_{2}\in\mathsf{X}$.
First define the extended probability distribution on $\mathsf{X}\times\mathsf{X}$
\[
\mu(x_{1},x_{2}):=\pi(x_{1})T_{2}(x_{1},x_{2})=\pi(x_{2})T_{2}(x_{2},x_{1}),
\]
for $(x_{1},x_{2})\in\mathsf{X}\times\mathsf{X}$. Setting $Qf(x_{1},x_{2}):=f(x_{2},x_{1})$,
we notice that reversibility of $T_{2}$ implies that $Q$ is an $\mu-$isometric
involution. Let $M_{2}\big((x_{1},x_{2});(y_{1},y_{2})\big):=\mathbb{I}\{y_{1}=x_{1}\}T_{2}(x_{1},y_{2})$
for $(x_{1},x_{2})\in\mathsf{X}\times\mathsf{X}$ and notice that
$M_{2}$ is $\mu-$reversible. The Markov chain of transition $P_{2}=QM_{2}$
is therefore $(\mu,Q)-$reversible from Proposition \ref{prop:QPsa}.
Following an idea of Liu \cite{liu1996peskun}, it is suggested in
\cite{neal2004improving} to use instead the transition $P_{1}=QM_{1}$,
where the $\mu-$reversible component $M_{2}$ is replaced with the
following ($\mu-$reversible) Metropolis-Hastings update,
\[
M_{1}((x_{1},x_{2});(y_{1},y_{2})):=\mathbb{I}\{y_{1}=x_{1}\}\Big[U\big((x_{1},x_{2});(y_{1},y_{2})\big)+\mathbb{I}\{y_{2}=x_{2}\}\big(1-U((x_{1},x_{2});\mathsf{X}\times\mathsf{X})\big)\Big],
\]
where
\[
U\big((x_{1},x_{2});(y_{1},y_{2})\big):=\frac{T_{2}(x_{1},y_{2})\mathbb{I}\{y_{2}\neq x_{2}\}}{1-T_{2}(x_{1},x_{2})}\min\left\{ 1,\frac{1-T_{2}(x_{1},x_{2})}{1-T_{2}(y_{1},y_{2})}\right\} .
\]
The $(\mu,Q)-$reversible kernel $P_{1}$ is designed so that backtracking,
the probability of returning to $x_{1}$ when sampling $y_{2}$ conditional
upon $x_{2}$, of the chain is reduced, compared to $P_{2}$. Let
$\{Z_{k},k\geq0\}$ denote a realization of the homogeneous Markov
chain of transition $P_{i}$ (for $i\in\{1,2\}$) and arbitrary initial
condition, one can check that its first component is a realisation
$\{X_{k},k\geq0\}$ of the Markov chain of transition $T_{i}$, and
in fact $Z_{k}=(X_{k},X_{k+1})$ for $k\geq0$. With an abuse of notation,
for any $\lambda\in[0,1)$ and $f\in L^{2}(\pi)$ we let ${\rm var}_{\lambda}\big(f,T_{1}\big):={\rm var}_{\lambda}\big(\breve{f},P_{1}\big)$
where for any $x_{1},x_{2}\in\mathsf{X}$, $\breve{f}(x_{1},x_{2}):=f(x_{1})$.
\begin{thm}
For any $g\in L^{2}(\mu)$ such that $Qg=g$, and $\lambda\in[0,1)$
we have ${\rm var}_{\lambda}\big(g,P_{1}\big)\leq{\rm var}_{\lambda}\big(g,P_{2}\big)$
and as a consequence, for any $f\in L^{2}(\pi)$
\[
{\rm var}_{\lambda}\big(f,T_{1}\big)\leq{\rm var}_{\lambda}\big(f,T_{2}\big).
\]
\end{thm}

\begin{proof}
Since for $(x_{1},x_{2})\neq(y_{1},y_{2})$
\begin{multline*}
M_{1}\big((x_{1},x_{2});(y_{1},y_{2})\big)=\mathbb{I}\{y_{1}=x_{1}\}T_{2}(x_{1},y_{2})\mathbb{I}\{y_{2}\neq x_{2}\}\min\left\{ \frac{1}{1-T_{2}(x_{1},x_{2})},\frac{1}{1-T_{2}(y_{1},y_{2})}\right\} \\
\geq\mathbb{I}\{y_{1}=x_{1}\}T_{2}(x_{1},y_{2}),
\end{multline*}
we deduce $QP_{1}\big((x_{1},x_{2});(y_{1},y_{2})\big)\geq QP_{2}\big((x_{1},x_{2});(y_{1},y_{2})\big)$
for $(x_{1},x_{2})\neq(y_{1},y_{2})$ and consequently from the identity
in (\ref{eq:dirichlet-form}) $\mathcal{E}(g,QP_{1})\geq\mathcal{E}(g,QP_{2})$
for any $g\in L^{2}(\mu)$. The first statement follows from Theorem
\ref{thm:ordering_discrete} (or Theorem \ref{thm:mdo-main-result}).
The second statement will follow by letting $g(x_{1},x_{2})=f(x_{1})+f(x_{2})$
for an arbitrary $f\in L^{2}(\pi)$ and once we have established that
for $\lambda\neq0$ and $i\in\{1,2\}$
\[
{\rm var}_{\lambda}\big(g,P_{i}\big)=-(1-\lambda^{2})\lambda^{-1}{\rm var}_{\pi}(f)+(1+\lambda)^{2}\lambda^{-1}{\rm var}_{\lambda}\big(f,T_{i}\big).
\]
Without loss of generality assume that $\pi(f)=0$. For both homogeneous
Markov chains of transitions $P_{1}$ and $P_{2}$ with initial condition
$Z_{0}=(X_{0},X_{1})\sim\mu$ we have $\mathbb{E}\big[g^{2}(Z_{0})\big]=2\mathbb{E}\big[f^{2}(X_{0})\big]+2\mathbb{E}\big[f(X_{0})f(X_{1})\big]$
and for $k\geq1$, with $Z_{k}=(X_{k},X_{k+1})$, we have
\begin{align*}
\mathbb{E}\big[g(Z_{0})g(Z_{k})\big] & =\mathbb{E}\big[\big(f(X_{0})+f(X_{1})\big)\big(f(X_{k})+f(X_{k+1})\big)\big]\\
 & =\mathbb{E}\big[f(X_{0})f(X_{k-1})\big]+2\mathbb{E}\big[f(X_{0})f(X_{k})\big]+\mathbb{E}\big[f(X_{0})f(X_{k+1})\big],
\end{align*}
therefore implying
\begin{align*}
\sum_{k\geq1}\lambda^{k}\mathbb{E}\big[g(Z_{0})g(Z_{k})\big] & =\lambda\sum_{k\geq0}\lambda^{k}\mathbb{E}\big[f(X_{0})f(X_{k})\big]+2\sum_{k\geq1}\lambda^{k}\mathbb{E}\big[f(X_{0})f(X_{k})\big]+\sum_{k\geq2}\lambda^{k-1}\mathbb{E}\big[f(X_{0})f(X_{k})\big]\\
 & =\lambda\mathbb{E}\big[f^{2}(X_{0})\big]+(\lambda^{2}+2\lambda)\mathbb{E}\big[f(X_{0})f(X_{1})\big]+(\lambda^{2}+2\lambda+1)\sum_{k\geq2}\lambda^{k-1}\mathbb{E}\big[f(X_{0})f(X_{k})\big].
\end{align*}
This yields for $i\in\{1,2\}$ and $\lambda\neq0$,
\begin{align*}
{\rm var}_{\lambda}\big(g,P_{i}\big) & =2(1+\lambda)\mathbb{E}\big[f^{2}(X_{0})\big]+2(\lambda^{2}+2\lambda+1)\mathbb{E}\big[f(X_{0})f(X_{1})\big]+2(1+\lambda)^{2}\sum_{k\geq2}\lambda^{k-1}\mathbb{E}\big[f(X_{0})f(X_{k})\big]\\
 & =-(1-\lambda^{2})\lambda^{-1}\mathbb{E}\big[f^{2}(X_{0})\big]+(1+\lambda)^{2}\lambda^{-1}\left(\mathbb{E}\big[f^{2}(X_{0})\big]+2\sum_{k\geq1}\lambda^{k}\mathbb{E}\big[f(X_{0})f(X_{k})\big]\right).
\end{align*}
Finally note that for $k\geq0$ $\mathbb{E}\big[\breve{f}(Z_{0})\breve{f}(Z_{k})\big]=\mathbb{E}\big[f(X_{0})f(X_{k})\big]$
and therefore ${\rm var}_{\lambda}\big(\breve{f},P_{1}\big)=\mathbb{E}\big[f^{2}(X_{0})\big]+2\sum_{k\geq1}\lambda^{k}\mathbb{E}\big[f(X_{0})f(X_{k})\big]$.
We can therefore conclude.
\end{proof}

\section{Continuous time scenario -- general results\label{sec:CT-theory}}

The continuous time scenario follows in part ideas similar to those
developed in the discrete time scenario, but requires the introduction
of the generator of the semigroup associated with the continuous time
process, leading to additional technical complications. In Subsection
\ref{subsec:continuous:Set-up-and-characterization} we develop a
crucial result of practical interest, Theorem \ref{prop:equiv-MDB-semigroup-generator},
which allows one to deduce that a (in general intractable) semigroup
is $(\mu,Q)-$self-adjoint when its generator is $(\mu,Q)-$symmetric
on a type of dense subset of its domain. In Subsection \ref{subsec:continuous-Ordering-of-asymptotic}
we establish the continuous time counterpart of Theorem \ref{thm:ordering_discrete},
that is show that ordering of tractable quantities involving the generators
of two $(\mu,Q)-$reversible processes implies an order on their asymptotic
variances (Theorem \ref{thm:continous-order-with-dirichlet}). We
remark that while establishing order rigorously may appear complex
and technical, checking the criterion suggesting order involves in
general elementary calculations. To the best of our knowledge no general
result is available in the continuous time reversible setup, that
is when $Q={\rm Id}$ in our setup, but note the works \cite{leisen2008extension,roberts2014minimising},
focused on particular scenarios.

\subsection{Set-up and characterization of $(\mu,Q)-$self-adjointness\label{subsec:continuous:Set-up-and-characterization}}

Let $\{Z_{t},t\geq0\}$ be a Markov process taking values in the space
$\mathsf{D}(\mathbb{R}_{+},E)$ of cadlag functions endowed with the
Skorokhod topology and corresponding probability space $\big(\Omega,\mathcal{F},\mathbb{P}\big)$.
We denote $\{P_{t},t\geq0\}$ the associated semi-group, assumed to
have an invariant distribution $\mu$ defined on $(E,\mathscr{E})$
and let $\big(\mathcal{D}^{2}\big(L,\mu\big),L\big)$ be the generator
associated with $\{P_{t},t\geq0\}$ i.e. $L$ and $\mathcal{D}^{2}\big(L,\mu\big)\subset L^{2}(\mu)$
are such that, with ${\rm Id}$ the identity operator,
\begin{equation}
\mathcal{D}^{2}\big(L,\mu\big):=\left\{ f\in L^{2}(\mu)\colon\lim_{t\downarrow0}\left\Vert t^{-1}\left(P_{t}-{\rm Id}\right)f-Lf\right\Vert _{\mu}=0\right\} .\label{eq:def-D2Lmu}
\end{equation}
From above $\{P_{t},t\geq0\}$ is a strongly continuous contraction,
$\mathcal{D}^{2}\big(L,\mu\big)$ is dense in $L^{2}(\mu)$ and $L$
is closed \cite[Corollary 1.6]{ethier2009markov}. For any $t\in\mathbb{R_{+}}$
we let $P_{t}^{*}$ denote the $L^{2}(\mu)-$adjoint of $P_{t}$,
and it is classical that $\{P_{t}^{*},t\geq0\}$ is a strongly continuous
contraction of invariant distribution $\mu$ and generator $\big(\mathcal{D}^{2}\big(L^{*},\mu\big),L^{*}\big)$,
the adjoint of $L$ \cite{pedersen2012analysis}, that is it holds
that for $f\in\mathcal{D}^{2}\big(L,\mu\big)$ and $g\in\mathcal{D}^{2}\big(L^{*},\mu\big)$,
$\bigl\langle Lf,g\bigr\rangle_{\mu}=\bigl\langle f,L^{*}g\bigr\rangle_{\mu}$.

In order to avoid repetition we group our basic assumptions on the
triplet $\big(\mu,Q,\{P_{t},t\geq0\}\big)$ used throughout this section.

\begin{hyp}\label{hyp:semi-group-strong-invariant}
\begin{enumerate}
\item $\mu$ is a probability distribution defined on $(E,\mathscr{E})$,
\item \label{enu:Pt-strongly-continuous}$\{P_{t},t\geq0\}$ is a strongly
continuous Markov semi-group of invariant distribution $\mu$,
\item \label{enu:Q-isometric-DL-invariant}$Q$ is a $\mu-$isometric involution.
\end{enumerate}
\end{hyp}
\begin{defn}
\label{def:Qself-adjoint-semigroup}We will say that the semi-group
$\{P_{t},t\geq0\}$ is $(\mu,Q)-$self-adjoint, if for all $f,g\in L^{2}(\mu)$
and $t\geq0$
\[
\bigl\langle P_{t}f,g\bigr\rangle_{\mu}=\bigl\langle f,QP_{t}Qg\bigr\rangle_{\mu}.
\]
\end{defn}

We aim to characterise the adjoint of the generator of a $(\mu,Q)-$self-adjoint
semigroup $\{P_{t},t\geq0\}$ and provide a practical simple condition
to establish this property for a given semigroup. We preface our first
results with a technical lemma. For two operators $\big(\mathcal{D}^{2}(A,\mu),A\big)$
and $\big(\mathcal{D}^{2}(B,\mu),B\big)$, $\mathcal{D}^{2}(AB,\mu):=\left\{ f\in\mathcal{D}^{2}(B,\mu)\colon Bf\in\mathcal{D}^{2}(A,\mu)\right\} $.
\begin{lem}
\label{lem:QPQ-semigroup}Let $\big(\mu,Q,\{P_{t},t\geq0\}\big)$
satisfying (A\ref{hyp:semi-group-strong-invariant}) and let $\{T_{t}:=QP_{t}Q,t\geq0\}$.
Then
\begin{enumerate}
\item $\big(\mu,Q,\{T_{t},t\geq0\}\big)$ satisfies (A\ref{hyp:semi-group-strong-invariant}),
\item the generator of $\{T_{t},t\geq0\}$ is $\big(\mathcal{D}^{2}\big(QLQ,\mu\big),QLQ\big)$.
\end{enumerate}
\end{lem}

\begin{proof}
From the properties of $Q$ and $\{P_{t},t\geq0\}$, it is immediate
that $\{T_{t},t\geq0\}$ is a semigroup leaving $\mu$ invariant.
Further for $f\in L^{2}(\mu)$ $\|QP_{t}Qf-f\|_{\mu}=\|P_{t}Qf-Qf\|_{\mu}$
from which the continuity follows. Denote $\big(\mathcal{D}^{2}(\tilde{L},\mu),\tilde{L}\big)$
the generator of $\{T_{t},t\geq0\}$. For $f\in\mathcal{D}^{2}(QLQ,\mu)$
we have $Qf\in\mathcal{D}^{2}(L,\mu)$ and therefore by (A\ref{hyp:semi-group-strong-invariant}),
\[
\lim_{t\downarrow0}\|t^{-1}(QP_{t}Q-{\rm Id})f-QLQf\|_{\mu}=\lim_{t\downarrow0}\|t^{-1}(P_{t}-{\rm Id})Qf-LQf\|_{\mu}=0,
\]
implying $\mathcal{D}^{2}(QLQ,\mu)\subset\mathcal{D}^{2}(\tilde{L},\mu)$
and $\tilde{L}f=QLQf$ for $f\in\mathcal{D}^{2}(QLQ,\mu)$. Similarly
for any $f\in\mathcal{D}^{2}(\tilde{L},\mu)$ 
\[
0=\lim_{t\downarrow0}\|t^{-1}(QP_{t}Q-{\rm Id})f-\tilde{L}f\|_{\mu}=\lim_{t\downarrow0}\|t^{-1}(P_{t}-{\rm Id})Qf-Q\tilde{L}QQf\|_{\mu},
\]
implying $Qf\in\mathcal{D}^{2}(L,\mu)$ and hence $f\in\mathcal{D}^{2}(QLQ,\mu)$.
We conclude.
\end{proof}
As a corollary one can characterise the generator of a $(\mu,Q)-$self-adjoint
semigroup.
\begin{prop}
\label{prop:properties-Qmu-SA-semigroup}Let $\big(\mu,Q,\{P_{t},t\geq0\}\big)$
satisfying (A\ref{hyp:semi-group-strong-invariant}) be $(\mu,Q)-$self-adjoint.
Then the generator of $\{P_{t}^{*},t\geq0\}$ is $\big(\mathcal{D}^{2}\big(QLQ,\mu\big),L^{*}=QLQ\big)$.
\end{prop}

\begin{proof}
We use Lemma \ref{lem:QPQ-semigroup} and the fact that here $P_{t}^{*}=QP_{t}Q$
for $t\geq0$.
\end{proof}
The following allows one to check $(\mu,Q)-$self-adjointness of a
semigroup from the restriction of its generator to a particular type
of dense subspace. A subspace $\mathcal{A}\subset\mathcal{D}^{2}\big(L,\mu\big)$
is said to be a core for $L$ if the closure of the restriction $L_{\mid\mathcal{A}}$
of $L$ to $\mathcal{A}$ is $L$, where the closure is to be taken
with respect to $\|(f,g)\|_{\mu}:=\|f\|_{\mu}+\|g\|_{\mu}$ for $f,g\in L^{2}(\mu)$
on the graph $\mathcal{G}(L)=\left\{ (f,Lf)\colon f\in\mathcal{D}^{2}\big(L,\mu\big)\right\} $.
\begin{thm}
\label{prop:equiv-MDB-semigroup-generator}Let $\big(\mu,\{P_{t},t\geq0\},Q\big)$
satisfying (A\ref{hyp:semi-group-strong-invariant}). Assume that
$\mathcal{A}$ is a core for $\big(L,\mathcal{D}^{2}(L,\mu)\big)$
such that
\begin{enumerate}
\item $f\in\mathcal{A}$ implies $Qf\in\mathcal{A}$,
\item for all $f,g\in\mathcal{A}$ we have $\bigl\langle Lf,g\bigr\rangle_{\mu}=\bigl\langle f,QLQg\bigr\rangle_{\mu},$
\end{enumerate}
then $\{P_{t},t\geq0\}$ is $(\mu,Q)-$self-adjoint.
\end{thm}

\begin{proof}
Since $\mathcal{A}$ is a core for $\big(\mathcal{D}^{2}(L,\mu),L\big)$,
$Q\mathcal{A}=\mathcal{A}$ and $Q$ is continuous, we have $\bigl\langle Lf,g\bigr\rangle_{\mu}=\bigl\langle f,QLQg\bigr\rangle_{\mu}$
for $f\in\mathcal{D}^{2}(L,\mu)$ and $g\in\mathcal{D}^{2}(QLQ,\mu)$.
Indeed, since $\mathcal{A}$ is a core for $L$, for any $f\in\mathcal{D}^{2}(L,\mu)$
there exists $\{f_{n}\in\mathcal{A},n\in\mathbb{N}\}$ such that $\lim_{n\rightarrow\infty}\|f_{n}-f\|_{\mu}+\|Lf_{n}-Lf\|_{\mu}=0$.
Similarly for $g\in\mathcal{D}^{2}(QLQ,\mu)$, then $Qg\in\mathcal{D}^{2}(L,\mu)$
and from the definition of a core one can find $\{\gamma_{n}\in\mathcal{A},n\in\mathbb{N}\}$
such that $\lim_{n\rightarrow\infty}\|\gamma_{n}-Qg\|_{\mu}+\|L\gamma_{n}-LQg\|_{\mu}=0$
implying $\lim_{n\rightarrow\infty}\|g_{n}-g\|_{\mu}+\|LQg_{n}-LQg\|_{\mu}=0$
with $\{g_{n}:=Q\gamma_{n}\in\mathcal{A},n\in\mathbb{N}\}$. Further
$\mathcal{D}^{2}(QLQ,\mu)\subset\mathcal{D}^{2}(L^{*},\mu)$ as for
any $g\in\mathcal{D}^{2}(QLQ,\mu)$ we have that $\mathcal{D}^{2}(L,\mu)\ni f\mapsto\bigl\langle Lf,g\bigr\rangle_{\mu}=\bigl\langle f,QLQg\bigr\rangle_{\mu}$
which is bounded and can be extended to $L^{2}(\mu)$ by density of
$\mathcal{D}^{2}(L,\mu)$, and we conclude by definition of the adjoint
\cite[Paragraph 5.1.2]{pedersen2012analysis}. From Lemma \ref{lem:QPQ-semigroup}
and the Hille-Yosida theorem, we have that for all $\lambda>0$ ${\rm Ran}(\lambda{\rm Id}-QLQ)=L^{2}(\mu)$.
Fix $\lambda>0$, then for any $g\in\mathcal{D}^{2}(L^{*},\mu)$ there
exists $h\in\mathcal{D}^{2}(QLQ,\mu)$ such that $(\lambda{\rm Id}-QLQ)h=(\lambda{\rm Id}-L^{*})g$
and hence for any $f\in\mathcal{D}^{2}(L,\mu)$
\[
\bigl\langle(\lambda{\rm Id}-L)f,g\bigr\rangle_{\mu}=\bigl\langle f,(\lambda{\rm Id}-L^{*})g\bigr\rangle_{\mu}=\bigl\langle f,(\lambda{\rm Id}-QLQ)h\bigr\rangle_{\mu}=\bigl\langle(\lambda{\rm Id}-L)f,h\bigr\rangle_{\mu}.
\]
Again from the Hille-Yosida theorem we can take $\lambda>0$ as above
and have ${\rm Ran}(\lambda{\rm Id}-L)=L^{2}(\mu)$, and the equality
above translates into $\bigl\langle k,g\bigr\rangle_{\mu}=\bigl\langle k,h\bigr\rangle_{\mu}$
for any $k\in L^{2}(\mu)$ and hence $g=h\in\mathcal{D}^{2}(QLQ,\mu)$,
$\mathcal{D}^{2}(L^{*},\mu)=\mathcal{D}^{2}(QLQ,\mu)$ and $L^{*}=QLQ$.
From the Duhamel formula we deduce that $P_{t}^{*}=QP_{t}Q$ for $t\geq0$.
\end{proof}

\subsection{Ordering of asymptotic variances\label{subsec:continuous-Ordering-of-asymptotic}}

For $f\in L^{2}(\mu)$ we are interested, when this quantity exists,
in the limit of
\[
{\rm var}(f,L):=\lim_{t\rightarrow\infty}{\rm var}\left(t^{-1/2}\int_{0}^{t}f(Z_{s}){\rm d}s\right),
\]
where $Z_{0}\sim\mu$. In some circumstances (for instance when a
Foster-Lyapunov function can be identified \cite[Theorem 4.3]{glynn1996})
the limit above exists and has the following expression
\[
{\rm var}(f,L)=2\bigl\langle f,Rf\bigr\rangle_{\mu},
\]
where $Rf:=\int_{0}^{+\infty}P_{t}f{\rm d}t$. For $\lambda>0$ and
$f\in L^{2}(\mu)$ we introduce ${\rm var}_{\lambda}(f,L):=2\bigl\langle f,R_{\lambda}f\bigr\rangle_{\mu}$,
where $R_{\lambda}$ is the bounded operator defined as
\[
R_{\lambda}f:=\int_{0}^{+\infty}\exp(-\lambda t)P_{t}f{\rm d}t,
\]
referred to as the resolvent from now on. It is classical that for
any $f\in L^{2}(\mu)$, $(\lambda{\rm Id}-L)R_{\lambda}f=f$ and for
$f\in\mathcal{D}(L,\mu)$, $R_{\lambda}(\lambda{\rm Id}-L)f=f$. As
in the discrete time setup, we leave the issue of checking whether
$\lim_{\lambda\downarrow0}{\rm var}_{\lambda}(f,L)={\rm var}(f,L)$
as separate. We note the following straightforward result.
\begin{lem}
If $\big(\mu,Q,\{P_{t},t\geq0\}\big)$ satisfies (A\ref{hyp:semi-group-strong-invariant})
and is $(\mu,Q)-$self-adjoint, then for any $\lambda>0$ the bounded
operator $R_{\lambda}$ is also $(\mu,Q)-$self-adjoint.
\end{lem}

For two semi-groups $\{P_{t,1},t\geq0\}$ and $\{P_{t,2},t\geq0\}$
leaving $\mu$ invariant and of generators $L_{1}$ and $L_{2}$ with
domains $\mathcal{D}^{2}(L_{1},\mu)$ and $\mathcal{D}^{2}(L_{2},\mu)$
we are interested in ordering ${\rm var}_{\lambda}(f,L_{1})$ and
${\rm var}_{\lambda}(f,L_{2})$ for $\lambda>0$. As in the discrete
time set-up the comparison relies on the Dirichlet forms, defined
as follows for a generator $L$ and $f\in\mathcal{D}^{2}(L,\mu)$,
\[
\mathcal{E}\big(f,L\big):=\bigl\langle f,-Lf\bigr\rangle_{\mu}.
\]
Our proof requires the introduction of interpolating processes, defined
at the level of their generators.

\begin{hyp}\label{hyp:semi-group-strong-invariant-interpolation}$\big(\mu,Q,\{P_{t,1},t\geq0\}\big)$
and $\big(\mu,Q,\{P_{t,2},t\geq0\}\big)$ satisfy (A\ref{hyp:semi-group-strong-invariant})
and are $(\mu,Q)-$self-adjoint. Their respective generators $\big(L_{1},\mathcal{D}^{2}(L_{1},\mu)\big)$
and $\big(L_{2},\mathcal{D}^{2}(L_{2},\mu)\big)$ are assumed
\begin{enumerate}
\item to have a common core $\mathcal{A}$ dense in $L^{2}(\mu)$ such that
$Q\mathcal{A\subset\mathcal{A}},$
\item to be such that for any $\beta\in[1,2]$ the operator $\big((2-\beta)L_{1}+(\beta-1)L_{2},\mathcal{D}^{2}(L_{1},\mu)\cap\mathcal{D}^{2}(L_{2},\mu)\big)$
\begin{enumerate}
\item has an extension defining a unique continuous contraction semigroup
$\{P_{t}(\beta),t\geq0\}$ on $L^{2}(\mu)$ of invariant distribution
$\mu$ and of (closed) generator $\big(L(\beta),\mathcal{D}^{2}(L(\beta),\mu)\big)$,
\item and for any $f\in\mathcal{A}$ we have $P_{t}(\beta)f\in\mathcal{A}$
for any $t\geq0$.
\end{enumerate}
\end{enumerate}
\end{hyp}

From \cite[Proposition 3.3]{ethier2009markov} the last assumption
and density of $\mathcal{A}$ in $L^{2}(\mu)$ imply that $\mathcal{A}$
is a core for $L(\beta)$, $\beta\in[1,2]$. Establishing that for
$\beta\in[1,2]$ the contraction semigroup $\{P_{t}(\beta),t\geq0\}$
exists may require one to resort to the Hille-Yosida theory and/or
perturbation theory results \cite{voigt1977perturbation,ethier2009markov},
but turns out to be straightforward in some scenarios such as those
treated in Section \ref{sec:Continuous-time-scenario-applications}.
For $\lambda>0$ and $\beta\in[1,2]$ we let $R_{\lambda}(\beta)$
be the corresponding resolvent operators. Differentiability of $\beta\rightarrow R_{\lambda}(\beta)f$
and the expression for the corresponding derivative are key to our
result, as is the case in the discrete time scenario. The right derivatives
of operators below are to be understood as limits in the Banach space
$L^{2}(\mu)$ equipped with the norm $\|\cdot\|_{\mu}$. We only state
the results for $f\in L_{2}(\mu)$ such that $Qf=f$ and note that
the case $Qf=-f$ is straightforward.
\begin{thm}
\label{thm:continous-order-with-dirichlet}Assume (A\ref{hyp:semi-group-strong-invariant-interpolation})
and that for any $\lambda>0,\beta\in[1,2]$ and $f\in\mathcal{A}$,
\begin{enumerate}
\item $R_{\lambda}(\beta)f\in\mathcal{D}^{2}(L_{1},\mu)\cap\mathcal{D}^{2}(L_{2},\mu)$
and there exists $\{g_{n}(\beta)\in\mathcal{A},n\in\mathbb{N}\}$
such that $\lim_{n\rightarrow\infty}(L_{1}-L_{2})g_{n}(\beta)=(L_{1}-L_{2})R_{\lambda}(\beta)f$,
\item $[1,2)\ni\beta\mapsto R_{\lambda}(\beta)$f is right differentiable
with
\[
\partial_{\beta}R_{\lambda}(\beta)f=R_{\lambda}(\beta)(L_{2}-L_{1})R_{\lambda}(\beta)f,
\]
and $\beta\mapsto\bigl\langle f,\partial_{\beta}R_{\lambda}(\beta)f\bigr\rangle_{\mu}$
is continuous,
\item either $\mathcal{E}\big(g,QL_{1}-QL_{2}\big)\geq0$ for any $g\in\mathcal{A}$
or $\mathcal{E}\big(g,L_{1}Q-L_{2}Q\big)\geq0$ for any $g\in\mathcal{A}$,
\end{enumerate}
then
\begin{enumerate}
\item for any $f\in\mathcal{A}$ satisfying $Qf=f$ and $\beta\in[1,2)$,
\begin{align*}
\partial_{\beta}\bigl\langle f,R_{\lambda}(\beta)f\bigr\rangle_{\mu} & =\mathcal{E}\big(QR_{\lambda}(\beta)f,L_{1}Q-L_{2}Q\big)=\mathcal{E}\big(R_{\lambda}(\beta)f,QL_{1}-QL_{2}\big)\geq0,
\end{align*}
\item for any $f\in L_{2}(\mu)$ such that $Qf=f$,
\[
{\rm var}_{\lambda}(f,L_{1})=2\bigl\langle f,R_{\lambda}(1)f\bigr\rangle_{\mu}\leq{\rm var}_{\lambda}(f,L_{2})=2\bigl\langle f,R_{\lambda}(2)f\bigr\rangle_{\mu}.
\]
\end{enumerate}
\end{thm}

\begin{proof}
For $\beta\in[1,2)$, $\delta\in(0,2-\beta]$ and $f\in\mathcal{A}$
such that $Qf=f$ we have $\delta^{-1}\big[\bigl\langle f,R_{\lambda}(\beta+\delta)f\bigr\rangle_{\mu}-\bigl\langle f,R_{\lambda}(\beta)f\bigr\rangle_{\mu}\big]=\bigl\langle f,\delta^{-1}\big[R_{\lambda}(\beta+\delta\big)-R_{\lambda}(\beta\big)\big]f\bigr\rangle_{\mu}$
and from the assumption we deduce
\begin{align*}
\partial_{\beta}\bigl\langle f,R_{\lambda}(\beta)f\bigr\rangle_{\mu} & =\bigl\langle f,\partial_{\beta}R_{\lambda}(\beta)f\bigr\rangle_{\mu}\\
 & =\bigl\langle f,R_{\lambda}(\beta)(L_{2}-L_{1})R_{\lambda}(\beta)f\bigr\rangle_{\mu}\\
 & =\bigl\langle R_{\lambda}^{*}(\beta)f,(L_{2}-L_{1})R_{\lambda}(\beta)f\bigr\rangle_{\mu}\\
 & =\bigl\langle QR_{\lambda}(\beta)Qf,(L_{2}-L_{1})R_{\lambda}(\beta)f\bigr\rangle_{\mu}\\
 & =\bigl\langle R_{\lambda}(\beta)f,-Q(L_{1}-L_{2})R_{\lambda}(\beta)f\bigr\rangle_{\mu},
\end{align*}
where we have used that $Qf=f$ and the fact that $Q$ is $\mu-$self-adjoint.
Using in addition that $Q^{2}={\rm Id}$, it is not difficult to establish
the alternate expression $\partial_{\beta}\bigl\langle f,R_{\lambda}(\beta)f\bigr\rangle_{\mu}=\bigl\langle QR_{\lambda}(\beta)f,-(L_{1}-L_{2})QQR_{\lambda}(\beta)f\bigr\rangle_{\mu}$.
The first claim follows from $\mathcal{E}\big(R_{\lambda}(\beta)f,QL_{1}-QL_{2}\big)=\lim_{n\rightarrow\infty}\mathcal{E}\big(g_{n},QL_{1}-QL_{2}\big)\geq0$.
The second claim follows from
\[
\bigl\langle f,R_{\lambda}(2)f\bigr\rangle_{\mu}-\bigl\langle f,R_{\lambda}(1)f\bigr\rangle_{\mu}=\int_{1}^{2}\mathcal{E}\big(R_{\lambda}(\beta)f,QL_{1}-QL_{2}\big){\rm d}\beta\geq0,
\]
continuity of $R_{\lambda}(1),R_{\lambda}(2)$ on $L^{2}(\mu)$ and
the density of $\mathcal{A}$ in $L^{2}(\mu)$.
\end{proof}
The following allows us to check the conditions of the theorem above.
\begin{lem}
\label{lem:continuous-diff-resolvent}Assume (A\ref{hyp:semi-group-strong-invariant-interpolation})
and that for any $\lambda>0$, $\beta\in[1,2]$ and $f\in\mathcal{A}$,
\begin{enumerate}
\item $t\mapsto(L_{2}-L_{1})P_{t}(\beta)f$ and $t\mapsto(L_{2}-L_{1})QP_{t}(\beta)f$
are continuous,
\item there exists $\delta(\beta)>0$ such that
\[
\left\{ \int_{0}^{\infty}\exp(-\lambda t)\|(L_{2}-L_{1})P_{t}(\beta)f\|_{\mu}{\rm d}t\right\} \vee\left\{ \sup_{|\beta'-\beta|\leq\delta(\beta)}\int_{0}^{\infty}\exp(-\lambda t)\|(L_{2}-L_{1})QP_{t}(\beta')f\|_{\mu}{\rm d}t\right\} <\infty.
\]
\end{enumerate}
Then for any $\beta\in[1,2)$ and $\lambda>0$, for any $f\in\mathcal{A}$,
\begin{enumerate}
\item $R_{\lambda}(\beta)f\in\mathcal{D}^{2}(L_{1},\mu)\cap\mathcal{D}^{2}(L_{2},\mu)$
and there exists $\{g_{n}(\beta)\in\mathcal{A},n\in\mathbb{N}\}$
such that $\lim_{n\rightarrow\infty}(L_{1}-L_{2})g_{n}(\beta)=(L_{1}-L_{2})R_{\lambda}(\beta)f$,
\item $[1,2)\ni\beta\mapsto R_{\lambda}(\beta)$f is right differentiable
with
\[
\partial_{\beta}R_{\lambda}(\beta)f=R_{\lambda}(\beta)(L_{2}-L_{1})R_{\lambda}(\beta)f,
\]
and $\beta\mapsto\bigl\langle f,\partial_{\beta}R_{\lambda}(\beta)f\bigr\rangle_{\mu}$
is continuous.
\end{enumerate}
\end{lem}

\begin{proof}
Let $f\in\mathcal{A}$, $\beta,\beta'\in[1,2]$ and $\lambda>0$.
Then for $t\geq0$ $P_{t}(\beta)f\in\mathcal{A}\subset\mathcal{D}(L(\beta'),\mu)$,
by definition $L(\beta')=L(\beta)+(\beta'-\beta)[L_{2}-L_{1}]$ on
$\mathcal{A}\subset\mathcal{D}^{2}(L_{1},\mu)\cap\mathcal{D}^{2}(L_{2},\mu)$
and $L(\beta)P_{t}(\beta)f=P_{t}(\beta)L(\beta)f$ from \cite[Proposition 1.5]{ethier2009markov}.
Therefore from the assumptions $t\mapsto\exp(-\lambda t)L(\beta')P_{t}(\beta)f$
is continuous and summable and from \cite[Lemma 1.4]{ethier2009markov}
we have 
\[
\int_{0}^{\infty}\exp(-\lambda t)L(\beta')P_{t}(\beta)f{\rm d}t=L(\beta')\int_{0}^{\infty}\exp(-\lambda t)P_{t}(\beta)f{\rm d}t,
\]
since $L(\beta')$ is closed, implying that $R_{\lambda}(\beta)f\in\mathcal{D}(L(\beta'),\mu)$.
Using the identity above for $\beta_{1}'\neq\beta_{2}'$ and taking
the difference we deduce that $R_{\lambda}(\beta)f\in\mathcal{D}^{2}(L_{1},\mu)\cap\mathcal{D}^{2}(L_{2},\mu)$
and that we can take $\{g_{n}(\beta)\in\mathcal{A},n\in\mathbb{N}\}$
such that for $n\in\mathbb{N}$
\[
g_{n}(\beta):=\frac{1}{n}\sum_{k=1}^{n^{2}}\exp(-\lambda k/n)P_{k/n}(\beta)f,
\]
in the first claim.  Now let $\beta\in[1,2)$ and $\delta\in\mathbb{R}$
such that $\beta+\delta\in[1,2]$. From the above we deduce that for
$f\in\mathcal{A}$,
\begin{align}
\big[R_{\lambda}(\beta+\delta\big)-R_{\lambda}(\beta\big)\big]f & =R_{\lambda}(\beta+\delta\big)\big[{\rm Id}-\big(\lambda{\rm Id}-L(\beta+\delta)\big)R_{\lambda}(\beta\big)\big]f\label{eq:continuity-resolvent-beta}\\
 & =R_{\lambda}(\beta+\delta\big)\big[\big(\lambda{\rm Id}-L(\beta)\big)-\big(\lambda{\rm Id}-L(\beta+\delta)\big)\big]R_{\lambda}(\beta\big)f\nonumber \\
 & =R_{\lambda}(\beta+\delta\big)\big[L(\beta+\delta)-L(\beta)\big]R_{\lambda}(\beta\big)f\nonumber \\
 & =\delta R_{\lambda}(\beta+\delta\big)\big[L_{2}-L_{1}\big]R_{\lambda}(\beta\big)f.\nonumber 
\end{align}
Using that $\vvvert R_{\lambda}(\beta\big)\vvvert_{\mu}\leq\lambda^{-1}$
for $\beta\in[1,2]$ we conclude that for any $f\in\mathcal{A}$ the
mapping $\beta\mapsto R_{\lambda}(\beta)f$ is continuous. Let $f\in\mathcal{A}$
and $\epsilon>0$, then from the density of $\mathcal{A}$ in $L^{2}(\mu)$,
there exists $g\in\mathcal{A}$ such that $\|\big[L_{2}-L_{1}\big]R_{\lambda}(\beta\big)f-g\|_{\mu}\leq\lambda\epsilon/4$,
and using the bound
\begin{multline*}
\|\big[R_{\lambda}(\beta+\delta\big)-R_{\lambda}(\beta\big)\big]\big[L_{2}-L_{1}\big]R_{\lambda}(\beta\big)f\|_{\mu}\leq\|\big[R_{\lambda}(\beta+\delta\big)-R_{\lambda}(\beta\big)\big]\big\{\big[L_{2}-L_{1}\big]R_{\lambda}(\beta\big)f-g\big\}\|_{\mu}\\
+\|\big[R_{\lambda}(\beta+\delta\big)-R_{\lambda}(\beta\big)\big]g\|_{\mu},
\end{multline*}
the fact that $\vvvert R_{\lambda}(\beta\big)\vvvert_{\mu}\leq\lambda^{-1}$
for $\beta\in[1,2]$ and the continuity of $\beta\mapsto R_{\lambda}(\beta)g$
we conclude that for $\delta$ sufficiently small, $\|\big[R_{\lambda}(\beta+\delta\big)-R_{\lambda}(\beta\big)\big]\big[L_{2}-L_{1}\big]R_{\lambda}(\beta\big)f\|_{\mu}\leq\epsilon$.
Therefore together with (\ref{eq:continuity-resolvent-beta}) we have
established that for any $\beta\in[1,2)$ and $f\in\mathcal{A}$ 
\[
\lim_{\delta\downarrow0}\|\delta^{-1}\big[R_{\lambda}(\beta+\delta\big)-R_{\lambda}(\beta\big)\big]f-R_{\lambda}(\beta\big)\big[L_{2}-L_{1}\big]R_{\lambda}(\beta\big)f\|_{\mu}=0.
\]
Further for $f\in\mathcal{A}$ such that $Qf=f$ and any $\beta,\beta'\in[1,2]$
\begin{multline*}
\bigl\langle R_{\lambda}(\beta')f,Q(L_{2}-L_{1})R_{\lambda}(\beta')f\bigr\rangle_{\mu}-\bigl\langle R_{\lambda}(\beta)f,Q(L_{2}-L_{1})R_{\lambda}(\beta)f\bigr\rangle_{\mu},\\
=\bigl\langle\big[R_{\lambda}(\beta')-R(\beta)\big]f,Q(L_{2}-L_{1})R_{\lambda}(\beta)f\bigr\rangle_{\mu}-\bigl\langle(L_{2}-L_{1})QR_{\lambda}(\beta')f,\big[R_{\lambda}(\beta)-R_{\lambda}(\beta')\big]f\bigr\rangle_{\mu}
\end{multline*}
and we conclude with the Cauchy-Schwarz inequality, another use of
\cite[Lemma 1.4]{ethier2009markov} and the continuity of $\beta\mapsto R_{\lambda}(\beta)f$.
\end{proof}

\section{Continuous time scenario -- example\label{sec:Continuous-time-scenario-applications}}

In this section we show how the results of the previous section can
be applied to a particular class of processes designed to perform
Monte Carlo simulation, which has recently received some attention
(Subsection \ref{subsec:PDMP-Monte-Carlo}). In Subsection \ref{subsec:PDMP-mu,Q-symmetry}
we establish that most processes considered in the literature are
indeed $(\mu,Q)-$self-adjoint--this includes in particular the Zig-Zag
(ZZ) process. In Subsection \ref{subsec:Zig-zag:-generator-properties}
we show that with some smoothness conditions on the intensities involved
in the definition of the ZZ process, then all the conditions required
to apply our general results, namely Theorem \ref{thm:continous-order-with-dirichlet}
and Lemma \ref{lem:continuous-diff-resolvent}, are satisfied. In
Subsection \ref{subsec:PDMP-Monte-Carlo} we apply our general theory
and present some applications. In addition we show how one can consider
more general versions of ZZ relying on nonsmooth intensities using
smooth approximation strategies which have the advantage of preserving
the correct invariant distribution.

\subsection{PDMP-Monte Carlo\label{subsec:PDMP-Monte-Carlo}}

We assume here that $E=\mathsf{X}\times\mathsf{V}$ and that the distribution
$\mu$ of interest has density (also denoted $\mu$)
\begin{equation}
\mu(x,v)\propto\exp\big(-U(x)\big)\varpi(v)\label{eq:def-joint-mu-distribution}
\end{equation}
with respect to some $\sigma-$finite measure denoted ${\rm d}(x,v)$,
where $U:\mathsf{X}=\mathbb{R}^{d}\rightarrow\mathbb{R}$ is an energy
function and $\varpi\colon\mathsf{V}\subset\mathbb{R}^{d}\rightarrow\mathbb{R}_{+}$
are such that $\mu$ induces a probability distribution. Piecewise
deterministic Markov processes (PDMPs) \cite{davis1993markov} are
continuous time processes with various applications in engineering
and science, but it has recently been shown \cite{faggionato2009non,peters2012rejection,2015bou-rabee-sanz-serna,bouchard2015bouncy,bierkens2015piecewise,bierkens2016zig}
that such processes can be used in order to sample from large classes
of distributions defined as above. The particular cases derived for
this purpose are known to be non-reversible, but we establish here
that they are in fact $(\mu,Q)-$reversible for a specific isometric
involution $Q$. This allows us to apply the theory developed in the
previous section and to compare their performance in terms of some
of their design parameters.

For $k\in\mathbb{Z}_{+}$ , for $i\in\left\llbracket 1,k\right\rrbracket $
define intensities $\lambda_{i}\colon E\rightarrow\mathbb{R}_{+}$,
$\lambda:=\sum_{i=1}^{k}\lambda_{i}$, for $(x,v)\in E$ and $t\geq0$
\[
\Lambda_{i}(t,x,v):=\int_{0}^{t}\lambda_{i}(x+uv,v){\rm d}u,
\]
 $\Lambda(t,x,v):=\sum_{i=1}^{k}\Lambda_{i}(t,x,v)$ and kernels $R_{i}:E\times\mathscr{E}\rightarrow[0,1]$
such that for any $(x,v)\in E$, $R_{i}\big((x,v),\{x\}\times\mathsf{V}\big)=1$.
For any $x\in\mathsf{X}$ and $i\in\left\llbracket 1,k\right\rrbracket $
we let $R_{x,i}\colon\mathsf{V}\times\mathscr{V}\rightarrow[0,1]$
be such that $R_{x,i}\big(v,A\big):=R_{i}\big((x,v),\{x\}\times A\big)$
for $(v,A)\in\mathsf{V}\times\mathscr{V}$. For $\varsigma_{1},\ldots,\varsigma_{k}\in\mathbb{R}_{+}$
we let $\mathcal{P}\big(\varsigma_{1},\ldots,\varsigma_{k}\big)$
denote the probability distribution of the random variable $M$ such
that $\mathbb{P}\big(M=m\big)\propto\varsigma_{m}$. The PDMPs of
interest here can be described algorithmically as in Algorithm \ref{alg:piecewise-deterministic}.

\begin{algorithm}
\rule[0.5ex]{1\textwidth}{1pt}
\begin{itemize}
\item Initialization $z(0)=\big(x(0),v(0)\big)$ , $T_{0}=0$ and $l=1$.
\item Repeat
\begin{enumerate}
\item Draw $T_{l}$ such that $\mathbb{P}\big(T_{l}\geq\tau\mid T_{l-1}\big)=\exp\left(-\Lambda(\tau-T_{l-1},X_{T_{l-1}},V_{T_{l-1}})\right)$,
\item $\big(X_{t},V_{t}\big)=\big(X_{T_{l-1}}+(t-T_{l-1})V_{T_{l-1}},V_{T_{l-1}}\big)$
for $t\in[T_{l-1},T_{l})$,
\item \label{enu:algo-change-direction-1}$X_{T_{l}}=\lim_{t\uparrow T_{l}}X_{t}$
and with $M\sim\mathcal{P}\big(\lambda_{1}(Z_{T_{l}}),\ldots,\lambda_{d}(Z_{T_{l}})\big)$
set $V_{T_{l}}\sim R_{X_{T_{l}},M}\big(V_{T_{l-1}},\cdot\big)$,
\item $l\leftarrow l+1$,
\end{enumerate}
\end{itemize}
\caption{A piecewise deterministic Markov process to sample from $\mu$.\label{alg:piecewise-deterministic}}

\rule[0.5ex]{1\textwidth}{1pt}
\end{algorithm}

Davis \cite{davis1993markov} (see also \cite{2018arXiv180705421D}
for an alternative construction) shows that this defines a process,
of corresponding semigroup $\{P_{t},t\in\mathbb{R}_{+}\}$, as soon
as the following standard two conditions on the intensity are satisfied
\cite[p. 62]{davis1993markov}: 

\begin{hyp}\label{hyp:properties-lambda}For $i\in\left\llbracket 1,k\right\rrbracket $,
\begin{enumerate}
\item $\lambda_{i}$ is measurable and $t\mapsto\lambda_{i}(x+tv,v)$ is
integrable for all $(x,v)\in E$,
\item \label{enu:pdmp-intensity-non-explosion}for any $t>0$ and $(x,v)\in E$,
$\mathbb{E}_{x,v}\big(N(t)\big)<\infty$, where $N(t):=\sum_{i=1}^{\infty}\mathbb{I}\{T_{i}\leq t\}$,
\end{enumerate}
\end{hyp}

\noindent Define for any $(x,v)\in E$ and $f\in\mathbb{R}^{E}$,
whenever the limit exists, 
\[
Df(x,v):=\lim_{h\rightarrow0}\frac{f(x+hv,v)-f(x,v)}{h},
\]
then the extended generator of the process above, which solves the
Martingale problem, is of the form
\begin{equation}
Lf:=Df+\sum_{i=1}^{k}\lambda_{i}\cdot\big[R_{i}f-f\big],\label{eq:pdmp-general-generator}
\end{equation}
for $f\in\mathcal{D}(L)$, a domain fully characterized by Davis \cite[Theorem 26.14, p. 69 and Remark 26.16]{davis1993markov}.

\noindent Let $\mathbf{M}(E)\subset\mathbb{R}^{E}$ be the set of
measurable functions and $\mathbf{B}(E)\subset\mathbf{M}(E)$ be the
set of bounded measurable functions. It can be shown that $\{P_{t},t\geq0\}$
is a contraction semigroup on $\mathbf{B}(E)$ equipped with the $\|\cdot\|_{\infty}$
norm. Further with 
\[
\mathbf{B}_{0}(E):=\left\{ f\in\mathbf{B}(E)\colon\lim_{t\downarrow0}\|P_{t}f-f\|_{\infty}=0\right\} ,
\]
one can show that $\{P_{t},t\geq0\}$ is a strongly continuous contraction
semigroup on $\mathbf{B}_{0}(E)$ \cite[p. 28-29]{davis1993markov}
of strong generator $\big(\mathcal{D}_{\infty}(L_{\infty}),L_{\infty}\big)$,
with $\mathcal{D}_{\infty}(L_{\infty})\subset\mathcal{D}(L)$ and
for any $f\in\mathcal{D}_{\infty}(L_{\infty})$, $L_{\infty}f=Lf$.

\noindent When $\mathsf{V}=\mathbb{R}^{d}$ (or such that $E$ is
a Riemannian sub-manifold) we define $\mathbf{C}(E):=\mathbf{C}^{0}(E):=\{f\in\mathbb{R}^{E}\colon f\text{ is continuous}\}$
and 
\[
\mathbf{C}^{1}(E):=\{f\in\mathbb{R}^{E}\colon f\text{ is continuously differentiable}\},
\]
let $\mathbf{C}_{c}(E)$ and $\mathbf{C}_{c}^{1}(E)$ be their restrictions
to compactly supported functions and $\mathbf{C}_{0}(E)\subset\mathbf{C}(E)$
the set containing functions vanishing at infinity. When $\mathsf{V}$
is finite we let, with $\mathbf{C}^{0}(\mathsf{X}):=\{f\in\mathbb{R}^{\mathsf{X}}\colon f\text{ is continuous}\}$
and for $i\in\mathbb{N}_{+}$ $\mathbf{C}^{i}(\mathsf{X}):=\{f\in\mathbb{R}^{\mathsf{X}}\colon f\text{ is \ensuremath{i} times continuously differentiable}\}$,
\[
\mathbf{C}^{i}(E):=\{f\in\mathbb{R}^{E}\colon\text{for any }v\in\mathsf{V},x\mapsto f(x,v)\in\mathbf{C}^{i}(\mathsf{X})\},
\]

\noindent use the simplified notation $\mathbf{C}(E):=\mathbf{C}^{0}(E)$,
and let $\mathbf{C}_{c}(E)$, $\mathbf{C}_{c}^{1}(E)$ be the corresponding
restrictions to functions $x\mapsto f(x,v)$ of compact support for
any $v\in\mathsf{V}$. We let $\mathbf{C}_{0}(E)$ be the set of $f\in\mathbf{C}(E)$
such that for any $\epsilon>0$ there exists $M\in\mathbb{R}_{+}$
such that $\vert f(x,v)\vert\leq\epsilon$ for $(x,v)\in B^{c}(0,M)\times\mathsf{V}$
where $B(0,M)=\{x\in\mathsf{X}\colon\|x\|\leq M\}$ and $\|\cdot\|$
is the Euclidian norm.

\subsection{$(\mu,Q)-$symmetry of some PDMP-Monte Carlo processes\label{subsec:PDMP-mu,Q-symmetry}}

\noindent From now on $Qf(x,v)=f(x,-v)$ for $f\in\mathbb{R}^{E}$
and $(x,v)\in E$. In the following we establish simple conditions
implying that $L$ is $(\mu,Q)-$symmetric on $\mathbf{C}_{c}^{1}(E)$,
which cover most known scenarios. Hereafter we will need the following
assumption on the potential $U$.

\begin{hyp}\label{hyp:ZZ-potential}The potential $U\colon\mathsf{X}\rightarrow\mathbb{R}$
is $\mathbf{C}^{2}(\mathsf{X})$ and
\[
\int\big[1+\|\nabla U(x)\|\big]\exp\big(-U(x)\big){\rm d}x<\infty.
\]

\end{hyp}

The following was shown in \cite[Proposition 3.2]{faggionato2009non}
for example.
\begin{lem}
\label{lem:adjointODEetc}Assume (A\ref{hyp:ZZ-potential}). Then
for $f,g\in\mathbf{C}_{c}^{1}(E)$ 
\begin{gather*}
\bigl\langle Df,g\bigr\rangle_{\mu}=\bigl\langle f,-Dg+DU\cdot g\bigr\rangle_{\mu},
\end{gather*}
 and $-Df=QDQ$f.
\end{lem}

\begin{proof}
Let $v\in\mathsf{V}$ and for $h\in\mathbb{R}^{E}$ define $x\mapsto h_{v}(x):=h(x,v)$
then, for $f,g\in\mathbf{C}_{c}^{1}(E)$, an integration by part gives
\[
\bigl\langle Df_{v},g_{v}\bigr\rangle_{\pi}=\bigl\langle f_{v},-Dg_{v}+DU\cdot g_{v}\bigr\rangle_{\pi},
\]
and we conclude by Fubini's theorem and integration with respect to
$\varpi$. The second statement follows from the fact that for $g\in\mathbf{C}_{c}^{1}(E)$
$Dg(x,v)=\bigl\langle\nabla g(x,v),v\bigr\rangle_{\mu}$ and therefore
for $g=Qf$, $DQf(x,v)=\bigl\langle\nabla f(x,-v),v\bigr\rangle_{\mu}$.
\end{proof}
The following establishes that a simple property on the family of
operators $\big\{ R_{i},i\in\left\llbracket 1,k\right\rrbracket \big\}$
ensures $(\mu,Q)-$symmetry of $L$, and hence invariance of $\mu$
if $\mathbf{C}_{c}^{1}(E)$ is a core.
\begin{thm}
\label{thm:pdmp-general-Q-self-adjointness}Let $\mu$ be a probability
distribution defined on $\big(E,\mathscr{E}\big)$ and consider the
semigroup $\{P_{t},t\geq0\}$ with extended generator $L$ given in
(\ref{eq:pdmp-general-generator}). Assume (A\ref{hyp:ZZ-potential}),
$\lambda-Q\lambda=DU$ and that for any $i\in\left\llbracket 1,k\right\rrbracket $
the operator $(\lambda_{i}\cdot R_{i})$ is $(\mu,Q)-$symmetric on
$\mathbf{C}_{c}^{1}(E)$. Then $L$ is $(\mu,Q)-$symmetric on $\mathbf{C}_{c}^{1}(E)$.
\end{thm}

\begin{proof}
Let $f,g\in\mathbf{C}_{c}^{1}(E)$. By assumption, for $i\in\left\llbracket 1,k\right\rrbracket $

\[
\bigl\langle(\lambda_{i}\cdot R_{i})f,g\bigr\rangle_{\mu}=\bigl\langle f,Q\lambda_{i}\cdot QR_{i}Qg\bigr\rangle_{\mu},
\]
and using $\lambda-Q\lambda=DU$ 
\begin{align*}
\bigl\langle\lambda\cdot f,g\bigr\rangle_{\mu} & =\bigl\langle f,\lambda\cdot g\bigr\rangle_{\mu}=\bigl\langle f,(Q\lambda+DU)\cdot g\bigr\rangle_{\mu}.
\end{align*}
Together with Lemma \ref{lem:adjointODEetc}, the above leads to
\[
\bigl\langle Lf,g\bigr\rangle_{\mu}=\bigl\langle f,QDQg\bigr\rangle_{\mu}+\bigl\langle f,DU\cdot g\bigr\rangle_{\mu}+\sum_{i=1}^{d}\bigl\langle f,Q\lambda_{i}\cdot QR_{i}Qg\bigr\rangle_{\mu}-\bigl\langle f,Q\lambda_{i}\cdot Q^{2}g\bigr\rangle_{\mu}-\bigl\langle f,DU\cdot g\bigr\rangle_{\mu},
\]
that is $\bigl\langle Lf,g\bigr\rangle_{\mu}=\bigl\langle f,QLQg\bigr\rangle_{\mu}$
and we conclude.
\end{proof}
\begin{rem}
With an abuse of notation, for $f\in\mathbb{R}^{\mathsf{V}}$ let
$Qf:=Q\breve{f}$ where for $(x,v)\in E$ $\breve{f}(x,v)=f(v)$.
For $f\in\mathbb{R}^{E}$ and for $x\in\mathsf{X}$ denote $f_{x}(\cdot)=f(x,\cdot)\colon\mathsf{V\rightarrow\mathbb{R}}$.
Let $i\in\left\llbracket 1,k\right\rrbracket $. If for any $x\in\mathsf{X}$,
$\lambda_{x,i}\cdot R_{x,i}$ is $(\varpi(\cdot),Q)-$symmetric on
$\mathbf{B}_{c}(\mathsf{V})$ then for $f,g\in\mathbf{C}_{c}^{1}(E)$
\end{rem}

\[
\bigl\langle(\lambda_{i}\cdot R_{i})f,g\bigr\rangle_{\mu}=\int\bigl\langle(\lambda_{i,x}\cdot R_{i,x})f_{x},g_{x}\bigr\rangle_{\varpi}\pi({\rm d}x)=\int\bigl\langle f_{x},Q(\lambda_{i,x}\cdot R_{i,x})Qg_{x}\bigr\rangle_{\varpi}\pi({\rm d}x)=\bigl\langle f,Q\lambda_{i}\cdot QR_{i}Qg\bigr\rangle_{\mu},
\]
that is $(\lambda_{i}\cdot R_{i})$ is $(\mu,Q)-$symmetric on $\mathbf{C}_{c}^{1}(E)$.

The most popular PDMP-MC processes satisfy the properties of Theorem
\ref{thm:pdmp-general-Q-self-adjointness} and are covered by the
following examples. For notational simplicity we may drop the index
$i$ below.
\begin{example}
\label{exa:general-update}Let $x\mapsto n(x)$ be a unit vector field
and assume that for any $(x,v)\in E$ we have $v-2\langle n(x),v\rangle n(x)\in\mathsf{V}$.
Consider the operator such that for any $f\in\mathbb{R}^{E}$ and
$(x,v)\in E$, $Rf(x,v):=f\big(x,v-2\langle n(x),v\rangle n(x)\big)$
and assume that the property $R\lambda=Q\lambda$ holds. Note that
$R^{2}={\rm Id}$ and that for any $f\in\mathbb{R}^{E}$ and $(x,v)\in E$,
$RQf(x,v)=f\big(x,-v+2\langle n(x),v\rangle n(x)\big)$ and hence
$QRQf=Rf$. Therefore for any $f,g\in\mathbf{C}_{c}^{1}(E)$,
\[
\bigl\langle(\lambda\cdot R)f,g\bigr\rangle_{\mu}=\bigl\langle Rf,\lambda\cdot g\bigr\rangle_{\mu}=\bigl\langle f,R\lambda\cdot Rg\bigr\rangle_{\mu}=\bigl\langle f,Q\lambda\cdot QRQg\bigr\rangle_{\mu}.
\]
Now let $\big\{ n_{i}\colon\mathsf{X}\rightarrow\mathbb{R}^{d},i\in\left\llbracket 1,k\right\rrbracket \big\}$
be unitary vector fields and $\big\{ a_{i}\colon\mathsf{X}\rightarrow\mathbb{R},i\in\left\llbracket 1,k\right\rrbracket \big\}$
such that $\nabla U=\sum_{i=1}^{k}a_{i}n_{i}$. Assume that for $i\in\left\llbracket 1,k\right\rrbracket $
the intensities are of the form $\lambda_{i}(x,v)=\varphi\big(a_{i}(x)\langle n_{i}(x),v\rangle\big)$
for $\varphi\colon\mathbb{R}\rightarrow\mathbb{R}_{+}$ such that
$\varphi(s)-\varphi(-s)=s$ and $R_{i}f(x,v):=f\big(x,v-2\langle n_{i}(x),v\rangle n_{i}(x)\big)$
for $f\in\mathbb{R}^{E}$ and $(x,v)\in E$. Possible choices of $\varphi$
are discussed later on and include $\varphi(s)=\max\{0,s\}$. Then
for $i\in\left\llbracket 1,k\right\rrbracket $, $R_{i}\lambda_{i}=Q\lambda_{i}$,
$(\lambda_{i}\cdot R_{i})$ is $(\mu,Q)-$symmetric on $\mathbf{C}_{c}^{1}(E)$
and $\lambda-Q\lambda=DU$. Therefore Theorem \ref{thm:pdmp-general-Q-self-adjointness}
holds. This covers the Zig-Zag and Bouncy Particle Sampler processes,
for example \cite{2018arXiv180808592A}.
\end{example}

\begin{example}
\label{exa:refresh-is-muQ}The choice $Rf(x,v)=\int f(x,w)\varpi({\rm d}w)$
for $(x,v)\in E$ and $f\in L^{2}(\mu)$, the ``refreshment'' operator,
is such that for any $f,g\in L^{2}(\mu)$, $\bigl\langle Rf,g\bigr\rangle_{\mu}=\bigl\langle f,Rg\bigr\rangle_{\mu}$,
$RQf=Rf$ and $QRf=Rf$ since for any $x\in\mathsf{X}$, $v\mapsto Rf(x,v)$
is constant. If for any $x\in\mathsf{X}$ the mapping $v\mapsto\bar{\lambda}(x,v)$
is constant (implying $\bar{\lambda}-Q\bar{\lambda}=0$) we deduce
that for any $f,g\in\mathbf{C}_{c}^{1}(E)$
\[
\bigl\langle(\bar{\lambda}\cdot R)f,g\bigr\rangle_{\mu}=\bigl\langle Rf,\bar{\lambda}\cdot g\bigr\rangle_{\mu}=\bigl\langle f,\bar{\lambda}\cdot Rg\bigr\rangle_{\mu}=\bigl\langle f,Q\bar{\lambda}\cdot QRQg\bigr\rangle_{\mu},
\]
that is $(\bar{\lambda}\cdot R)$ is $(\mu,Q)-$symmetric on $\mathbf{C}_{c}^{1}(E)$.
In fact, from the proof of Lemma \ref{lem:horowitz-update-Q-adjoint}
we note that $R$ can be taken to be Horowitz's refreshment operator.
\end{example}

\begin{example}
The choice $R_{x}f_{x}(v)\propto\int f_{x}(w)Q\lambda(x,w)\varpi({\rm d}w)$,
with $R_{x}\mathbf{1}(v)=1$ when possible, for any $(x,v)\in E$
and $f\in\mathbf{C}_{c}^{1}(E)$ has been suggested in \cite{fearnhead2016piecewise}.
It is such that for $f,g\in\mathbf{C}_{c}^{1}(E)$ and $x\in\mathsf{X}$,
\begin{align*}
\int f_{x}(v)g_{x}(w)\lambda(x,v)Q\lambda(x,w)\varpi({\rm d}w)\varpi({\rm d}v) & =\int Qf_{x}(v)g_{x}(w)Q\lambda(x,v)Q\lambda(x,w)\varpi({\rm d}w)\varpi({\rm d}v),
\end{align*}
and we conclude that $(\lambda_{x}\cdot R_{x})$ is $(\varpi,Q)-$self-adjoint.
\end{example}

\begin{rem}
We note that Theorem \ref{thm:pdmp-general-Q-self-adjointness} holds
more generally when the operator $D$ is replaced with the generator
$D_{F}$ of a dynamic with time-reversal symmetry \cite{lamb1998time}
for which $\bigl\langle D_{F}f,g\bigr\rangle_{\mu}=\bigl\langle f,QD_{F}Qg+D_{F}U\cdot g\bigr\rangle_{\mu}$,
which is the case for the Liouville operator for an arbitrary potential
$H(x,v)$, and the condition on the total intensity rate adjusted
accordingly. We do not pursue this here for brevity.
\end{rem}

\subsection{Zig-Zag: generator and semigroup properties\label{subsec:Zig-zag:-generator-properties}}

Zig-Zag (ZZ) is a particular continuous time Markov process designed
to sample from $\mu$ and described in Alg. \ref{alg:piecewise-deterministic}.
The name was coined in \cite{bierkens2015piecewise} and further extended
in \cite{bierkens2016zig}, and can be interpreted as being a particular
case of the process studied in \cite{faggionato2009non}. In this
scenario $k=d+1$ , $\mathsf{V}:=\{-1,1\}^{d}$, $\varpi$ is the
uniform distribution and, with $\{\mathbf{e}_{i}\in\mathbb{R}^{d},i\in\left\llbracket 1,d\right\rrbracket \}$
the canonical basis of $\mathbb{R}^{d}$, for $i\in\left\llbracket 1,d\right\rrbracket $
and $(x,v)\in E$ we let $R_{i}f(x,v):=f(x,v-2v_{i}\mathbf{e}_{i})$
where $v_{i}:=\left\langle v,\mathbf{e}_{i}\right\rangle $. Note
that this corresponds to $n_{i}(x)=\mathbf{e}_{i}$ in Example \ref{exa:general-update}.
For $i=d+1$ we let $\lambda_{d+1}(x,v)=\bar{\lambda}$ for $\bar{\lambda}\in\mathbb{R}_{+}$
and $R_{d+1}$ is as in Example \ref{exa:refresh-is-muQ}. We require
the following assumptions on the intensities.

\begin{hyp}\label{hyp:ZZ-smooth-intensities}For any $i\in\left\llbracket 1,d\right\rrbracket $
and $(x,v)\in E$ we have
\begin{enumerate}
\item $\lambda_{i}\in\mathbf{C}^{1}(E)$ and $\lambda_{i}>0$,
\item $\lambda_{i}(x,v)-Q\lambda_{i}(x,v)=\partial_{i}U(x)v_{i}$,
\item $R_{i}\lambda_{i}(x,v)=Q\lambda_{i}(x,v)$.
\end{enumerate}
\end{hyp}

The following establishes the existence of such intensities.
\begin{prop}
\label{cor:regularized-intensities}Assume (A\ref{hyp:ZZ-potential}).
Let $\phi:\mathbb{R}\rightarrow[0,1]$ be such that $r\phi(r^{-1})=\phi(r)$
for $r\geq0$ and define for any $(x,v)\in E$ and $i\in\left\llbracket 1,d\right\rrbracket $,
\[
\lambda_{i}^{\phi}(x,v):=-\log\left(\phi\big(\exp(\partial_{i}U(x)v_{i})\big)\right)\geq0.
\]
If further $\phi<1$ and $\phi\in\mathbf{C}^{1}(\mathbb{R})$ then
$\{\lambda_{i},i\in\left\llbracket 1,d\right\rrbracket \}$ satisfies
(A\ref{hyp:ZZ-smooth-intensities}).
\end{prop}

\begin{proof}
The first property is direct. For the second property, from the assumption
on $\phi$, we have
\begin{align*}
\lambda_{i}^{\phi}(x,v)-\lambda_{i}^{\phi}(x,-v) & =\lambda_{i}^{\phi}(x,v)-\partial_{i}U(x)(-v_{i})+\log\left(\phi\big(\exp(\partial_{i}U(x)v_{i}\big))\right)\\
 & =\lambda_{i}^{\phi}(x,v)+\partial_{i}U(x)v_{i}-\lambda_{i}^{\phi}(x,v).
\end{align*}
The last property was established in Example \ref{exa:general-update}
(here $n_{i}=\mathbf{e}_{i}$ for $i\in\left\llbracket 1,d\right\rrbracket $).
\end{proof}
\begin{cor}
The choice $\phi(r)=r/(1+r)$ satisfies the assumptions of Proposition
\ref{cor:regularized-intensities}, but this is not the case for the
canonical choice $\phi(r)=\min\{1,r\}$.
\end{cor}

We now establish properties required of $\{P_{t},t\geq0\}$ and its
generator in order to check (A\ref{hyp:semi-group-strong-invariant-interpolation})
and apply Theorem \ref{thm:continous-order-with-dirichlet} and Lemma
\ref{lem:continuous-diff-resolvent}.
\begin{prop}
\label{prop:ZZ:Qself}Let $L$ be the extended generator of the ZZ
process and assume (A\ref{hyp:ZZ-potential})-(A\ref{hyp:ZZ-smooth-intensities}).
Then $L$ is $\big(\mu,Q\big)-$symmetric on $\mathbf{C}_{c}^{1}(E)$.
\end{prop}

\begin{proof}
This follows the discussion of Example \ref{exa:general-update} and
application of Theorem \ref{thm:pdmp-general-Q-self-adjointness}.
\end{proof}
For $\mathbf{A},\mathbf{B}\subset\mathbf{M}(E)$ and $t>0$ we let
$P_{t}\mathbf{A}\subset\mathbf{B}$ mean that for any $f\in\mathbf{A}$
such that $P_{t}f$ exists, then $P_{t}f\in\mathbf{B}$. A semigroup
$\{P_{t},t\geq0\}$ is said to be Feller if $\mathbf{C}_{0}(\mathsf{E})\subset\mathbf{B}_{0}(\mathsf{E})$
and $\{P_{t},t\geq0\}$ is a strongly continuous contraction semigroup
on $\mathbf{C}_{0}(\mathsf{E})$ equipped with $\|\cdot\|_{\infty}$,
that is for any $s,t>0$, $P_{s+t}=P_{s}P_{t}$, $\|P_{t}f\|_{\infty}\leq\|f\|_{\infty}$,
$P_{t}\mathbf{C}_{0}(\mathsf{E})\subset\mathbf{C}_{0}(\mathsf{E})$
and for any $f\in\mathbf{C}_{0}(\mathsf{E})$, $\lim_{t\downarrow0}\|P_{t}f-f\|_{\infty}=0$.
\begin{thm}
\label{thm:ZZ-core}Consider a ZZ process of intensities $\big\{\lambda_{i},i\in\left\llbracket 1,d+1\right\rrbracket \big\}$
satisfying (A\ref{hyp:ZZ-smooth-intensities}). Then
\begin{enumerate}
\item $\{P_{t},t\geq0\}$ is Feller,
\item for any $t>0$, $P_{t}\mathbf{C}_{c}^{1}(E)\subset\mathbf{C}_{c}^{1}(E)$,
\item $\mathbf{C}_{c}^{1}(E)$ is a core for the strong generator of the
Feller semigroup $\{P_{t},t\geq0\}$,
\item $\mu$ is invariant for $\{P_{t},t\geq0\}$,
\item $\{P_{t},t\geq0\}$ can be extended to a strongly continuous semigroup
on $L^{2}(\mu)$ equipped with $\|\cdot\|_{\mu}$,
\item $\mathbf{C}_{c}^{1}(E)$ is a core for the strong generator of the
extended semigroup on $L^{2}(\mu)$.
\end{enumerate}
\end{thm}

\begin{proof}
The proof is an adaptation of \cite[Proposition 15 and Lemma 17 ]{2018arXiv180705421D}
to the present ZZ scenario for which $\mathsf{V}$ consists of a finite
set of bounded velocities. Note that due to the discrete nature of
$\mathsf{V}$ gradients of the form $\nabla_{x,v}f(x,v)$ (for suitable
functions) appearing in the statements of \cite{2018arXiv180705421D}
should be replaced throughout with derivatives with respect to the
position only, that is $\nabla_{x}f(x,v)$. Establishing \cite[Lemma 17 and Corollary 19]{2018arXiv180705421D}
requires checking that the ZZ process is non-explosive and that its
characteristics satisfy in particular conditions \cite[(A2) and (A3)]{2018arXiv180705421D}
which we refer to as $\mathbf{\left\langle A2\right\rangle }$ and
$\mathbf{\left\langle A3\right\rangle }$ in the present manuscript
in order to avoid confusion. For convenience we provide simplified
formulations, adapted to the specific ZZ process considered here,
of $\mathbf{\left\langle A2\right\rangle }$ and $\mathbf{\left\langle A3\right\rangle }$
and statements of \cite[Proposition 15 and Lemma 17 ]{2018arXiv180705421D}
in Appendix \ref{sec:durmus-results} where we also check that these
hold.

Once the conditions for \cite[Lemma 17 and Corollary 19]{2018arXiv180705421D}
to hold are satisfied the reasoning is as follows. \cite[Proposition 15 (b)]{2018arXiv180705421D}
allows us to conclude that $\mathbf{C}_{c}^{1}(E)\subset\mathcal{D}_{\infty}(L_{\infty})$
and \cite[Lemma 17]{2018arXiv180705421D} implies that for $i\in\{0,1\}$
and $t\geq0$, $P_{t}\mathbf{C}^{i}(E)\subset\mathbf{C}^{i}(E)$.
Let $f\in\mathbf{C}_{0}(E)$, $\epsilon>0$ and $M>0$ be such that
$\vert f(x,v)\vert\leq\epsilon$ for $(x,v)\in B^{c}(0,M)\times\mathsf{V}$.
Then for $(x,v)\in B^{c}(0,M+d^{1/2}t)\times\mathsf{V}$, ${\rm supp}\left\{ P_{t}(x,v;\cdot)\right\} \subset B^{c}(0,M)\times\mathsf{V}$
and therefore $\vert P_{t}f(x,v)\vert\leq\epsilon$ and with \cite[Lemma 17]{2018arXiv180705421D}
for $i=0$, we deduce $P_{t}f\in\mathbf{C}_{0}(E)$. Therefore for
any $t\geq0$ $P_{t}\mathbf{C}_{0}(E)\subset\mathbf{C}_{0}(E)$ and
from \cite[Proposition 15 (a)]{2018arXiv180705421D} $\mathbf{C}_{0}(E)\subset\mathbf{B}_{0}(E)$
from which we conclude that $\{P_{t},t\geq0\}$ is a strongly continuous
contraction semigroup on $\mathbf{C}_{0}(\mathsf{E})$ equipped with
$\|\cdot\|_{\infty}$, that is $\{P_{t},t\geq0\}$ is Feller. Let
$f\in\mathbf{C}_{c}^{1}(E)$, then using the reasoning above with
$\epsilon=0$ and \cite[Lemma 17]{2018arXiv180705421D} for $i=1$
we deduce that for any $t\geq0$ $P_{t}f\in\mathbf{C}_{c}^{1}(E)$,
that is for any $t\geq0$ $P_{t}\mathbf{C}_{c}^{1}(E)\subset\mathbf{C}_{c}^{1}(E)$.
This stability property, together with the Feller property and density
of $\mathbf{C}_{c}^{1}(E)$ in $\mathbf{C}_{0}(E)$ allows us to conclude
that $\mathbf{C}_{c}^{1}(E)$ is a core for the strong generator $L_{\infty}$
of $\{P_{t},t\geq0\}$ as a semigroup defined on $\mathbf{C}_{0}(E)$
equipped with $\|\cdot\|_{\infty}$ \cite[Proposition 3.3]{ethier2009markov}.
From Proposition \ref{prop:ZZ:Qself} and \cite[Proposition 9.2]{ethier2009markov}
we conclude that $\mu$ is invariant for $\{P_{t},t\geq0\}$. By density
of $\mathbf{C}_{0}(E)$ in $L^{2}(\mu)$ we conclude that for any
$t\geq0$ $P_{t}$ can be extended to a linear bounded operator on
$L^{2}(\mu)$ (we use the same notation for simplicity) and $\{P_{t},t\geq0\}$
defines a strongly continuous semigroup on $L^{2}(\mu)$ equipped
with $\|\cdot\|_{\mu}$. Since for $f\in\mathbf{C}_{c}^{1}(E)$ we
have 
\[
\lim_{t\downarrow0}\|t^{-1}(P_{t}f-f)-Lf\|_{\mu}\leq\lim_{t\downarrow0}\|t^{-1}(P_{t}f-f)-Lf\|_{\infty}=0,
\]
we deduce that $\mathbf{C}_{c}^{1}(E)\subset\mathcal{D}^{2}(L_{\mu},\mu)$.
Finally, since $\mathbf{C}_{c}^{1}(E)$ is dense in $L^{2}(\mu)$
and $P_{t}\mathbf{C}_{c}^{1}(E)\subset\mathbf{C}_{c}^{1}(E)$ we can
apply \cite[Proposition 3.3]{ethier2009markov} from which we deduce
that $\mathbf{C}_{c}^{1}(E)$ is a core for $\big(L_{\mu},\mathcal{D}^{2}(L_{\mu},\mu)\big)$.
\end{proof}

\subsection{Zig-Zag - main results and some examples\label{subsec:ZZ-example}}

The main result of this section is the ordering of Theorem \ref{thm:ZZ-main-result},
which we illustrate with two examples.
\begin{thm}
\label{thm:ZZ-main-result}Assume (A\ref{hyp:ZZ-potential}) and consider
two ZZ processes of intensities $\big\{\lambda_{1,i},i\in\left\llbracket 1,d+1\right\rrbracket \big\}$
and $\big\{\lambda_{2,i},i\in\left\llbracket 1,d+1\right\rrbracket \big\}$
satisfying (A\ref{hyp:ZZ-smooth-intensities}) such that for $i\in\left\llbracket 1,d+1\right\rrbracket $,
$\|\lambda_{1,i}-\lambda_{2,i}\|_{\mu}<\infty$. Then if for all $g\in\mathbf{C}_{c}^{1}(E)$,
\begin{equation}
\bigl\langle g,-(L_{1}-L_{2})Qg\bigr\rangle_{\mu}=\sum_{i=1}^{d+1}\bigl\langle g,(\lambda_{1,i}-\lambda_{2,i})\cdot[{\rm Id}-R_{i}Q]g\bigr\rangle_{\mu}-\bigl\langle g,(\lambda_{1}-\lambda_{2})\cdot[{\rm Id}-Q]g\bigr\rangle_{\mu}\geq0,\label{eq:ZZ-diff-dirichlet}
\end{equation}
then ${\rm var}_{\lambda}(L_{1},f)\leq{\rm var}_{\lambda}(L_{2},f)$
for $\lambda>0$ and $f\in L^{2}(\mu)$ such that $Qf=f$.
\end{thm}

\begin{rem}
The assumption on the intensity is satisfied as soon as for some $c,C>0$,
for any $i\in\left\llbracket 1,d+1\right\rrbracket $ and $(x,v)\in E$,
$\lambda_{1,i}(x,v)\leq c+C\|\nabla U(x)\|$ and $\int\|\nabla U\|^{2}{\rm d}\pi<\infty$.
As we shall see this can be checked for various examples.
\end{rem}

\begin{proof}
We check that the assumptions in (A\ref{hyp:semi-group-strong-invariant-interpolation})
are satisfied. First we note that for any $\beta\in[1,2]$ one can
define a ZZ process with intensities $(2-\beta)\lambda_{1,i}+(\beta-1)\lambda_{2,i}$
for $i\in\left\llbracket 1,d+1\right\rrbracket $ and extended generator
$(2-\beta)L_{1}+(\beta-1)L_{2}$. Further note that for $i\in\left\llbracket 1,d+1\right\rrbracket $,
$(2-\beta)\lambda_{1,i}+(\beta-1)\lambda_{2,i}$ satisfies (A\ref{hyp:ZZ-smooth-intensities}).
From Theorem \ref{thm:ZZ-core} for $\beta\in[1,2]$ we have that
$\mathbf{C}_{c}^{1}(E)\subset\mathcal{D}^{2}(L_{\mu}(\beta),\mu)$
is a core for $L_{\mu}(\beta)$, the generator of $\{P_{t}(\beta),t\geq0\}$
on $L^{2}(\mu)$, dense in $L^{2}(\mu)$ and such that for any $t\geq0$,
$P_{t}(\beta)\mathbf{C}_{c}^{1}(E)\subset\mathbf{C}_{c}^{1}(E)$.
From the definition of $\mathbf{C}_{c}^{1}(E)$, $Q\mathbf{C}_{c}^{1}(E)\subset\mathbf{C}_{c}^{1}(E)$
and from Proposition \ref{prop:ZZ:Qself} $L_{\mu}(\beta)$ is $(\mu,Q)-$symmetric
on $\mathbf{C}_{c}^{1}(E)$ for any $\beta\in[1,2]$. One can therefore
apply Theorem \ref{prop:equiv-MDB-semigroup-generator} and deduce
that for any $\beta\in[1,2]$, $\{P_{t}(\beta),t\geq0\}$ is $(\mu,Q)-$self-adjoint.
We now turn to checking the assumptions of Lemma \ref{lem:continuous-diff-resolvent}.
First note that we have for $g\in\mathbf{C}_{c}^{1}(E)$
\[
-(L_{1}-L_{2})g=\sum_{i=1}^{d+1}(\lambda_{1,i}-\lambda_{2,i})\cdot[{\rm Id}-R_{i}]g,
\]
where for $i\in\left\llbracket 1,d+1\right\rrbracket $
\begin{align*}
\|(\lambda_{1,i}-\lambda_{2,i})\cdot[{\rm Id}-R_{i}]g\|_{\mu} & \leq\|\lambda_{1,i}-\lambda_{2,i}\|_{\mu}\|[{\rm Id}-R_{i}]g\|_{\infty}\\
 & \leq2\|\lambda_{1,i}-\lambda_{2,i}\|_{\mu}\|g\|_{\infty},
\end{align*}
with $\|\lambda_{1,i}-\lambda_{2,i}\|_{\mu}<\infty$ by assumption.
Now for $\beta\in[1,2]$, $f\in\mathbf{C}_{c}^{1}(E)$ and $t\in\mathbb{R_{+}}$,
we choose $g_{t}=P_{t}(\beta)f\in\mathbf{C}_{c}^{1}(E)$ or $g_{t}=QP_{t}(\beta)f\in\mathbf{C}_{c}^{1}(E)$
(from Theorem \ref{thm:ZZ-core}), note that in both scenarios $\|g_{t}\|_{\infty}\leq\|f\|_{\infty}$
and deduce the continuity and summability conditions of Lemma \ref{lem:continuous-diff-resolvent}.
We conclude with Theorem \ref{thm:continous-order-with-dirichlet}.
\end{proof}
\begin{example}
\label{exa:BierkDun}When $d=1$, $R_{1}=Q$ and therefore $R_{1}Q={\rm Id}$.
If we further assume that $\lambda_{1,2}=\lambda_{2,2}=\bar{\lambda}$
then for all $g\in\mathbf{C}_{c}^{1}(E)$
\[
\bigl\langle g,-(L_{1}-L_{2})Qg\bigr\rangle_{\mu}=-\bigl\langle g,(\lambda_{1,1}-\lambda_{2,2})\cdot[{\rm Id}-Q]g\bigr\rangle_{\mu}\geq0,
\]
whenever $\lambda_{2,1}\geq\lambda_{1,1}$, a result similar to that
of \cite{bierkens2016limit}.
\end{example}

The situation where the total event rate is constant, that is $\lambda_{1}=\lambda_{2}$
in the expression above, but distributed differently between updates
of the velocity leads to the following.
\begin{example}
\label{exa:constant-total-rate}Let $d=2$ and for $g\in\mathbf{C}_{c}^{1}(E)$
let
\[
Lg=Dg+\sum_{i=1}^{2}\lambda_{i}\cdot\big[R_{i}-{\rm Id}\big]g
\]
and consider the ZZ processes of generators, for $\mathbf{C}^{1}(\mathsf{X})\ni\bar{\gamma}\colon\mathsf{X}\rightarrow\mathbb{R}_{+}$,
\begin{align*}
L_{1}g & =Lg+\bar{\gamma}/2\cdot\sum_{i=1}^{2}\big[R_{i}-{\rm Id}\big]g\\
L_{2}g & =Lg+\bar{\gamma}\cdot\big[\varPi-{\rm Id}\big]g.
\end{align*}
Then for $g\in\mathbf{C}_{c}^{1}(E)$,
\begin{align*}
\bigl\langle g,-(L_{1}-L_{2})Qg\bigr\rangle_{\mu} & =\bigl\langle g,\bar{\gamma}/2\cdot\sum_{i=1}^{2}[{\rm Id}-R_{i}Q]g\bigr\rangle_{\mu}-\bigl\langle g,\bar{\gamma}\cdot[{\rm Id}-\varPi]g\bigr\rangle_{\mu}\\
 & =\bigl\langle g,\bar{\gamma}/2\cdot\sum_{i=1}^{2}[{\rm Id}-R_{i}]g\bigr\rangle_{\mu}-\bigl\langle g,\bar{\gamma}\cdot[{\rm Id}-\varPi]g\bigr\rangle_{\mu}\geq0,
\end{align*}
where the equality follows from $R_{1}Q=R_{2}$, $R_{2}Q=R_{1}$,
Lemma \ref{lem:ZZ-polarization} and \cite[p. 52]{o2014analysis}.
We therefore conclude that in this setup partial refreshment of the
velocity is superior to full refreshment in terms of asymptotic variance. 
\end{example}

We note that checking (\ref{eq:ZZ-diff-dirichlet}) involves the difference
of two non-negative terms (from Lemma \ref{lem:ZZ-polarization})
and may be challenging to establish for this class of processes. For
example we have not been able to extend the result of Example \ref{exa:BierkDun}
to the situation where $d\geq2$, yet. We have not explored comparisons
involving other updates $R_{i}$, which would require establishing
Theorem \ref{thm:ZZ-core} for this setup, and rather focus on the
following issue. Intensities of interest may not satisfy (A\ref{hyp:ZZ-smooth-intensities})
and we may not be able to apply Theorem \ref{thm:ZZ-core}. This is
the case for the so-called canonical choice $\lambda(x,v)=\big(\partial U(x)v\big)_{+}$,
which may however be of interest as suggested by the following. In
the following discussion we assume $d=1$ for presentational simplicity,
but the approach is valid for $d\geq1$.
\begin{prop}
Let $\lambda\colon E\rightarrow\mathbb{R}_{+}$ be an intensity satisfying
$\lambda(x,v)-Q\lambda(x,v)=\partial U(x)v$, then $\lambda(x,v)\geq\big(\partial U(x)v\big)_{+}$.
\end{prop}

\begin{proof}
For $x\in\mathsf{X}$ consider the sets $V_{\pm}(x)=\left\{ v\in\mathsf{V}\colon\pm\partial U(x)v\geq0\right\} $.
From the assumption, for $(x,v)\in E$ 
\[
\lambda(x,v)-Q\lambda(x,v)=\big(\partial U(x)v\big)_{+}-\big(-\partial U(x)v\big)_{+},
\]
and we deduce that for $v\in V_{+}(x)$
\[
\lambda(x,\pm v)=\lambda(x,\mp v)+\big(\pm\partial U(x)v\big)_{+}\geq\big(\pm\partial U(x)v\big)_{+},
\]
and conclude since $\mathsf{V}=\{-1,1\}$.
\end{proof}
A natural question is whether we can establish that the choice $\lambda^{0}(x,v):=\big(\partial U(x)v\big)_{+}$
is optimum in terms of asymptotic variance. Our argument relies on
the existence of regularizing intensities satisfying the following
properties.

\begin{hyp}\label{hyp:regularized-intensities}The family of intensities
$\left\{ \lambda^{\epsilon},\epsilon\geq0\right\} $ satisfies for
any $(x,v)\in E$ and $\epsilon>0$,
\begin{enumerate}
\item $\lambda^{\epsilon}\in\mathbf{C}^{1}(E)$ and $\lambda^{\epsilon}>0$,
\item $\epsilon\mapsto\lambda^{\epsilon}(x,v)$ is non-increasing,
\item $\lambda^{\epsilon}(x,v)-Q\lambda^{\epsilon}(x,v)=\partial U(x)v,$
\item \label{enu:regular-lambda-unif-cv}$\lim_{\epsilon\downarrow0}\sup_{(x,v)\in E}\vert\lambda^{\epsilon}(x,v)-\lambda^{0}(x,v)\vert=0$.
\end{enumerate}
\end{hyp}

Intensities satisfying these properties exist:
\begin{prop}
\label{prop:example-regularized-intensities}Assume (A\ref{hyp:ZZ-potential})
and for any $\epsilon>0$ define the intensities such that for $(x,v)\in E$,
\[
\lambda^{\epsilon}(x,v):=-\log\left(\phi_{\epsilon}\big(\exp(\partial U(x)v)\big)\right),
\]
where for $r>0$
\[
\phi_{\epsilon}(r):=r[1-\Phi(\epsilon/2+\log(r)/\epsilon)]+[1-\Phi(\epsilon/2-\log(r)/\epsilon)]>0,
\]
with $\Phi(\cdot)$ the cumulative distribution function of the $\mathcal{N}(0,1)$.
Then $\left\{ \lambda^{\epsilon},\epsilon>0\right\} $ satisfies (A\ref{hyp:regularized-intensities}).
\end{prop}

\begin{proof}
As shown in Proposition \ref{prop:bound-approx-penalty-intensity}
$\phi_{\epsilon}(r)$ is the acceptance probability of the penalty
method \cite{ceperley1999penalty}, a particular instance of the Metropolis-Hastings
algorithm, and therefore satisfies the assumptions of Proposition
\ref{cor:regularized-intensities}. Condition (A\ref{hyp:regularized-intensities})-\ref{enu:regular-lambda-unif-cv}
is a consequence of the corollary of Proposition \ref{prop:bound-approx-penalty-intensity}.
\end{proof}
\begin{thm}
Let $d=1$, assume (A\ref{hyp:ZZ-potential}) and $\int\|\nabla U\|^{2}{\rm d}\pi<\infty$,
and consider two ZZ processes of common invariant distribution and
of intensities $\lambda_{1}$ and $\lambda_{2}$ where $\lambda_{1}(x,v):=\big(\partial U(x)\cdot v\big)_{+}$
and $\lambda_{2}(x,v):=\lambda_{1}(x,v)+\gamma(x,v)$ with $0\leq\gamma\leq c+C\|\nabla U\|$
for $c,C>0$, $\gamma\in\mathbf{C}^{1}(E)$ and such that $\gamma-Q\gamma=0$.
Then for any $f\in L^{2}(\mu)$ such that $Qf=f$ and $\lambda\in[0,1)$
\[
{\rm var}_{\lambda}(f,L_{1})\leq{\rm var}_{\lambda}(f,L_{2}).
\]
\end{thm}

\begin{proof}
We consider regularized intensities $\lambda_{2}^{\epsilon}(x,v):=\lambda_{1}^{\epsilon}(x,v)+\gamma(x,v)$
satisfying (A\ref{hyp:regularized-intensities}) (we have shown that
we can construct such intensities in Proposition \ref{prop:example-regularized-intensities}).
From Example \ref{exa:BierkDun} and Theorem \ref{thm:ZZ-main-result},
for any $f\in L^{2}(\mu)$ such that $Qf=f$ and $\epsilon>0$ we
have
\[
{\rm var}_{\lambda}(f,L_{1}^{\epsilon})\leq{\rm var}_{\lambda}(f,L_{2}^{\epsilon}).
\]
We can now conclude with Theorem \ref{thm:cv-regularized-semigroup-mu-invariant}
and Lemma \ref{lem:cv-regularized-asympvar}.
\end{proof}
One can consider more general forms for $\lambda_{2}$ in the theorem
above. For example the result will hold when $\gamma$ can be uniformly
approximated by a sequence $\{\gamma^{\epsilon}\in\mathbf{C}^{1}(E),\epsilon>0\}$
such that $\gamma^{\epsilon}\geq0$ for $\epsilon_{0}>\epsilon>0$
for some $\epsilon_{0}>0$. Another possibility is to consider generalizations
of the ideas of Proposition \ref{prop:example-regularized-intensities}:
for example with $\tilde{\phi}(r)=r/(1+r)$ instead of $\phi(r)=\min\{1,r\}$
as a starting point in Proposition \ref{prop:bound-approx-penalty-intensity}
one can analogously define a family of acceptance ratios which is
automatically such that $\tilde{\phi}_{\epsilon}(r)\leq\phi_{\epsilon}(r)$
for $\epsilon>0$ and $r\geq0$, define the corresponding intensities,
and then proceed as above to compare the processes with intensities
derived from $\tilde{\phi}(\cdot)$ and $\phi(\cdot)$. 

\noindent 

\section{Conclusion}

We have extended the set of practical tools available to characterize
existing Markov chain or process Monte Carlo algorithms to the scenario
where the building blocks involved are not reversible. We have shown
how they can be used to characterize algorithms previously beyond
the reach of earlier theory, confirming in some cases their good properties
in full generality. A natural question, not addressed here, is that
of the comparison of speed of convergence to equilibrium for the class
of processes considered here, for which a partial result exists in
the time-reversible setup (see Theorem \ref{thm:carraciolo-tierney}).
Consider for example the scenario where $\pi$ has finite support
$\llbracket1,d\rrbracket$ and the transition in (\ref{eq:toy-gustafson})
is combined into a $2-$cycle (cf. Subsection \ref{subsec:equivalence})
with the $(\mu,Q)-$reversible transition $P'_{\alpha}=(1-\alpha){\rm Id}+\alpha Q$
for $\alpha\in[0,1]$, as suggested in \cite{diaconis2000analysis}
(see also references therein). The analysis of \cite{diaconis2000analysis}
(see also \cite{GADE2007382}), in the situation where $\pi$ is the
uniform distribution, shows that near optimal convergence speed to
equilibrium is achieved for $\alpha(d)=c/d>0$, whereas application
of Theorem \ref{thm:mdo-extension-main-result} shows that $\alpha$
closer to zero is a better choice when asymptotic variance is of interest,
since $\mathcal{E}(g,QP'_{\alpha})=(1-\alpha)/2\int[g(x,v)-g(x,-v)]^{2}\mu\big({\rm d}(x,v)\big)$.
To the best of our knowledge no systematic spectral theory exists
in the setup considered in this manuscript, despite the numerous analogies
with the $\mu-$self-adjoint scenario and its practical interest.
We note the very recent work \cite{2019arXiv190501691B} (focused
on a restricted scenario) and \cite{2018arXiv180808592A} (which provides
lower bounds on the spectral gap) which both suggest difficulties
and the need for the development of new tools.

\section{Acknowledgements}

The authors acknowledge support from EPSRC ``Intractable Likelihood:
New Challenges from Modern Applications (ILike)'', (EP/K014463/1).
CA acknowledges support from EPSRC ``COmputational Statistical INference
for Engineering and Security (CoSInES)'', (EP/R034710/1).

\newpage{}

\appendix

\section{Key results of \cite{2018arXiv180705421D} \label{sec:durmus-results}}

For the reader's convenience we reformulate the results of \cite{2018arXiv180705421D}
for the specific scenario where the flow used is linear, that is we
consider the PDMP of extended generator
\[
Lf=Df+\lambda[R-{\rm Id}]f,
\]
where here
\[
Rf=\lambda^{-1}\sum_{i=1}^{d}\lambda_{i}R_{i}f.
\]
Following the terminology of \cite{2018arXiv180705421D}, we refer
to $\big(\lambda,R\big)$ as characteristics of the process. The deterministic
flow for this process is $\varphi_{t}(x,v)=(x+tv,v)$ and the associated
regularity condition required in \cite{2018arXiv180705421D} translates
to: for any $v\in\mathsf{V}$, $(x,t)\mapsto x+tv$ is continuously
differentiable (which is trivially satisfied). For brevity we omit
this property from the statements of \cite{2018arXiv180705421D},
leading to:
\begin{lyxlist}{00.00.0000}
\item [{$\mathbf{\left\langle A2\right\rangle }$}] Let $\{P_{t},t\geq0\}$
be a non-explosive PDMP semigroup with characteristics $(\lambda,R)$.
Assume that for any $T\in\mathbb{R}_{+}$ there exists $M\in\mathbb{R}_{+}$
such that for all $(x,v)\in E$ and $t\in[0,T]$, ${\rm supp}\big(P_{t}(x,v;\cdot)\big)\subset B(x,M)\times\mathsf{V}$.
\item [{$\mathbf{\left\langle A3\right\rangle }$}] The characteristics
$(\lambda,R)$ satisfy
\begin{enumerate}
\item for all compact sets $K\subset E$ and $T\in\mathbb{R}_{+}$ there
exists a compact set $\tilde{K}\subset E$ such that for all $n\in\mathbb{N_{+}}$
and $\left\{ t_{i}\in\mathbb{R}_{+},i\in\left\llbracket 1,n\right\rrbracket \right\} $
such that $\sum_{i=1}^{n}t_{i}\leq T$, there exists a sequence of
compact sets $\left\{ K_{i}\subset E,i\in\left\llbracket 0,n\right\rrbracket \right\} $
such that $\bigcup{}_{i=0}^{n}K_{i}\subset\tilde{K}$ with $K_{0}:=K$
and
\begin{enumerate}
\item $K_{i}$, $i\in\left\llbracket 1,n\right\rrbracket $ dependent on
$\left\{ t_{i}\in\mathbb{R}_{+},j\in\left\llbracket 1,i\right\rrbracket \right\} $,
\item for $i\in\left\llbracket 0,n-1\right\rrbracket $, $s_{i}\in[0,t_{i+1}]$
and $s_{n+1}\in[0,T-\sum_{i=1}^{n}t_{j}]$
\[
\bigcup_{(x,v)\in K_{i}}{\rm supp}\left\{ R(x+t_{i+1}v,v;\cdot)\right\} \subset K_{i+1},
\]
and for any $(x,v)\in K_{i}$, then $(x+s_{i+1}v,v)\in\tilde{K}$.
\end{enumerate}
\item $\lambda\in\mathbf{C}^{1}(E)$, $(\lambda R)\mathbf{C}^{1}(E)\subset\mathbf{C}^{1}(E)$
and there exists a locally bounded function $\Psi\colon E\rightarrow\mathbb{R}_{+}$
such that for all $(x,v)\in K$, $K$ compact, and $f\in\mathbf{C}^{1}(E)$
\[
\|\nabla_{x}\big(\lambda Rf\big)(x,v)\|\leq\sup_{(x,v)\in K}\|\Psi(x,v)\|_{\infty}\sup\left\{ \vert f(y,w)+\|\nabla_{x}f(y,w)\|\colon(y,w)\in{\rm supp}\left\{ R(x,v;\cdot)\right\} \vert\right\} .
\]
\end{enumerate}
\end{lyxlist}
As pointed out in \cite{2018arXiv180705421D}, the first property
(which the authors refer to as compact compatibility) implies that
the process is non-explosive. The two results of \cite{2018arXiv180705421D}
we use are:
\begin{prop*}[{\cite[Proposition 15]{2018arXiv180705421D}}]
Let $\{P_{t},t\geq0\}$ be a PDMP semigroup of characteristics $(\lambda,R)$
satisfying $\mathbf{\left\langle A2\right\rangle }$. Then
\begin{enumerate}
\item $\mathbf{C}_{0}(E)\subset\mathbf{B}_{0}(E)$,
\item if $\lambda\in\mathbf{C}(E)$ and $(\lambda R)\mathbf{C}_{c}(E)\subset\mathbf{C}_{c}(E)$
then $\mathbf{C}_{c}^{1}(E)\subset\mathcal{D}_{\infty}(L_{\infty})$.
\end{enumerate}
\end{prop*}
\begin{lem*}[{\cite[Lemma 17]{2018arXiv180705421D}}]
Let $\{P_{t},t\geq0\}$ be a non-explosive PDMP semigroup of characteristics
$(\lambda,R)$ satisfying $\mathbf{\left\langle A3\right\rangle }$.
Then for $i\in\{0,1\}$, $t\geq0$ $P_{t}\mathbf{C}^{i}(E)\subset\mathbf{C}^{i}(E)$.
\end{lem*}
We now establish that these results hold for the ZZ process considered
here.

Checking $\mathbf{\left\langle A2\right\rangle }$ is identical to
the argument \cite[Proposition 23]{2018arXiv180705421D} concerned
with the scenario where the velocity is bounded, since for any $t\in\mathbb{R}_{+}$
and $(x,v)\in E$, $P_{t}(x,v;B(x,\sqrt{d}t)\times\mathsf{V})=1$.
We turn to checking the conditions of \cite[Proposition 15]{2018arXiv180705421D}.
From (A\ref{hyp:ZZ-smooth-intensities}) $\lambda\in\mathbf{C}^{1}(E)$
and $\lambda>0$ and for any $f\in\mathbf{C}_{c}(E)$ then $Rf(x,v)\in\mathbf{C}_{c}(E)$
since for $i\in\left\llbracket 1,d\right\rrbracket $, $R_{i}f\in\mathbf{C}_{c}(E)$.
We conclude.

Establishing that $\mathbf{\left\langle A3\right\rangle }$-(a) holds
is identical to the result in the proof of \cite[Proposition 23]{2018arXiv180705421D}
concerned with the linear flow and bounded velocities. We repeat their
argument adapted to the present scenario: for all $M\in\mathbb{R}_{+}$
let $B(0,M)\subset\mathsf{X}$ then for all $(x,v)\in B(0,M)\times\mathsf{V}$,
$R(x,v;B(0,M)\times\mathsf{V})=1$. Therefore, for any $K\subset B(0,M_{K})\times\mathsf{V}$
(with $M_{K}$ such that the projection of $K$ on $\mathsf{X}$ is
contained in $B(0,M_{K})$) and $T\in\mathbb{R}_{+}$ $\mathbf{\left\langle A3\right\rangle }$-(a)
is satisfied with $\tilde{K}=B(0,M_{K}+\sqrt{d}T)\times\mathsf{V}$
and $K_{i}:=B(0,M_{K}+\sqrt{d}\sum_{j=0}^{i}t_{j})\times\mathsf{V}$
$i\in\left\llbracket 1,n\right\rrbracket $. As a result the process
is non-explosive. Finally we check that $\mathbf{\left\langle A3\right\rangle }$-(b)
holds. For $f\in\mathbf{C}^{1}(E)$
\begin{align*}
\|\nabla_{x}\big(\lambda Rf\big)(x,v) & \|\leq\sum_{i=1}^{d+1}\vert R_{i}f(x,v)\vert\|\nabla_{x}\lambda_{i}(x,v)\|+\lambda_{i}(x,v)\|\nabla_{x}\big(R_{i}f\big)(x,v)\|\\
 & \leq\left(\lambda(x,v)\vee\sum_{i=1}^{d+1}\|\nabla_{x}\lambda_{i}(x,v)\|\right)\sup\left\{ \vert f(y,w)\vert+\|\nabla_{x}f(y,w)\|,(y,w)\in{\rm supp}\{R(x,v;\cdot)\}\right\} .
\end{align*}
From (A\ref{hyp:ZZ-smooth-intensities}) $(x,v)\mapsto\lambda(x,v)\vee\sum_{i=1}^{d+1}\|\nabla_{x}\lambda_{i}(x,v)\|$
is locally bounded. We conclude that \cite[Lemma 17]{2018arXiv180705421D}
holds.

\section{Expression for Dirichlet forms}

\noindent  Computation of $\mathcal{E}\big(f,LQ\big)$ requires
computation of terms of the form $\bigl\langle f,\lambda\cdot[{\rm Id}-R]Qf\bigr\rangle_{\mu}$.
The identity
\[
\bigl\langle f,\lambda\cdot[{\rm Id}-R]Qf\bigr\rangle_{\mu}=\bigl\langle f,\lambda\cdot[{\rm Id}-RQ]f\bigr\rangle_{\mu}-\bigl\langle f,\lambda\cdot[{\rm Id}-Q]f\bigr\rangle_{\mu}
\]
motivates the following result.
\begin{lem}
\label{lem:ZZ-polarization}Assume that the operator $R$ is such
that for any $x\in\mathsf{X}$, $\lambda_{x}\cdot R_{x}$ is $(\varpi,Q)-$symmetric
on $\mathcal{D}\big(R_{x}\big)$. Then for any $f\in L^{2}(\mu)$
such that for any $x\in\mathsf{X}$ $f_{x}\in\mathcal{D}\big(R_{x}\big)$
and the integral exists, we have
\[
\bigl\langle f,\lambda\cdot[{\rm Id}-RQ]f\bigr\rangle_{\mu}=\frac{1}{2}\int\big(f(x,v)-f(x,w)\big)^{2}\lambda(x,v)\mu\big({\rm d}(x,v)\big)R_{x}Q\big(v,{\rm d}w\big).
\]
\end{lem}

\begin{cor}
Note that $R_{x}={\rm Id}$ is $(\varpi,Q)-$symmetric and therefore
\[
\bigl\langle f,\lambda\cdot[{\rm Id}-Q]f\bigr\rangle_{\mu}=\frac{1}{2}\int\big(f(x,v)-Qf(x,v)\big)^{2}\lambda(x,v)\mu\big({\rm d}(x,v)\big).
\]
\end{cor}

\begin{proof}
By assumption $\bigl\langle\lambda\cdot RQf,g\bigr\rangle_{\mu}=\bigl\langle Qf,Q(\lambda\cdot R)Qg\bigr\rangle_{\mu}=\bigl\langle f,\lambda\cdot RQg\bigr\rangle_{\mu}$,
that is $\lambda\cdot RQ$ is symmetric. Now we use polarization
\begin{align*}
\bigl\langle f,\lambda\cdot[{\rm Id}-RQ]f\bigr\rangle_{\mu} & =\frac{1}{2}\int\left[2f^{2}(x,v)+\big(f(x,v)-f(x,w)\big)^{2}-f^{2}(x,v)-f^{2}(x,w)\right]\lambda(x,v)\mu\big({\rm d}(x,v)\big)R_{x}Q\big(v,{\rm d}w\big)\\
 & =\frac{1}{2}\left\{ \int\big(f(x,v)-f(x,w)\big)^{2}\lambda(x,v)\mu\big({\rm d}(x,v)\big)R_{x}Q\big(v,{\rm d}w\big)\right.\\
 & \hspace{5cm}\left.-\int f^{2}(x,v)\big[\lambda(x,v)-\lambda(x,v)\big]\mu\big({\rm d}(x,v)\big)\right\} ,
\end{align*}
where, with $(x,v)\mapsto\mathbf{1}(x,v)=1$ we have used $\bigl\langle\mathbf{1},\lambda\cdot RQf^{2}\bigr\rangle_{\mu}=\bigl\langle\lambda,f^{2}\bigr\rangle_{\mu}$
since $\lambda\cdot RQ$ is symmetric.
\end{proof}

\section{Continuity of $\epsilon\protect\mapsto{\rm var}_{\lambda}(f,L_{\epsilon})$
for }
\begin{lem}
\label{lem:cv-regularized-asympvar}For any $\epsilon\geq0$ let $\left\{ P_{t}^{\epsilon},t\geq0\right\} $
be a semigroup on $L^{2}(\mu)$ leaving $\mu$ invariant, or generator
$\big(L_{\epsilon},\mathcal{D}^{2}(L_{\epsilon},\mu)\big)$ and assume
that for any $t\geq0$ and $f\in L^{2}(\mu)$
\[
\lim_{\epsilon\downarrow0}\|P_{t}^{\epsilon}f-P_{t}^{0}f\|_{\mu}=0.
\]
Then for any $f\in L^{2}(\mu)$ and $\lambda>0$,
\[
\lim_{\epsilon\downarrow0}{\rm var}_{\lambda}(f,L_{\epsilon})={\rm var}_{\lambda}(f,L_{0}).
\]
\end{lem}

\begin{proof}
For $\lambda>0$ and $f\in L^{2}(\mu)$ we have
\begin{align*}
\vert{\rm var}_{\lambda}(f,L_{\epsilon})-{\rm var}_{\lambda}(f,L_{0})\vert & \leq\int\int\exp(-\lambda t)\vert f(x,v)\vert\vert P_{t}^{\epsilon}f(x,v)-P_{t}^{0}f(x,v)\vert\mu\big({\rm d}(x,v)\big){\rm d}t,\\
 & \leq\|f\|_{\mu}\int\exp(-\lambda t)\|P_{t}^{\epsilon}f-P_{t}^{0}f\|_{\mu}{\rm d}t,
\end{align*}
from the Cauchy-Schwarz inequality. Since for $t\geq0$ $\|P_{t}^{\epsilon}f-P_{t}^{0}f\|_{\mu}\leq2\|f\|_{\mu}$
we can apply the dominated convergence theorem and conclude.
\end{proof}
\begin{thm}
\label{thm:cv-regularized-semigroup-mu-invariant}Let $d=1$. For
any $\epsilon>0$, let $\left\{ P_{t}^{\epsilon},t\geq0\right\} $
be a ZZ process of intensity As in Proposition \ref{prop:example-regularized-intensities}.
Then, with $\left\{ P_{t},t\geq0\right\} $ the semigroup of the ZZ
process using canonical intensities,
\begin{enumerate}
\item for any $f\in\mathbf{B}(E)$, $(x,v)\in E$ any $t\geq0$ and $\epsilon>0$
such that $1-\big[\exp(\epsilon)-1\big]^{1/2}>0$, 
\[
\vert P_{t}f(x,v)-P_{t}^{\epsilon}f(x,v)\vert\leq-2\log\left\{ 1-\big[\exp(\epsilon)-1\big]^{1/2}\right\} t\|f\|_{\infty},
\]
\item $\{P_{t},t\geq0\}$ is Feller and $\mathbf{C}_{c}^{1}(E)$ is a core
for the corresponding strong generator,
\item $\mu$ is invariant for $\{P_{t},t\geq0\}$,
\item $\{P_{t},t\geq0\}$ can be extended to a strongly continuous semigroup
on $L^{2}(\mu)$ equipped with $\|\cdot\|_{\mu}$.
\item For any $f\in L_{2}(\mu)$ and $t\geq0$ we have
\[
\lim_{\epsilon\downarrow0}\|P_{t}f-P_{t}^{\epsilon}f\|_{\mu}=0.
\]
\end{enumerate}
\end{thm}

\begin{proof}
For any $\epsilon>0$ (A\ref{hyp:regularized-intensities}) implies
(A\ref{hyp:ZZ-smooth-intensities}) from which we deduce that the
conclusions of Theorem \ref{thm:ZZ-core} hold for $\left\{ P_{t}^{\epsilon},t\in\mathbb{R}_{+}\right\} $.
Further from Proposition \ref{prop:bound-approx-penalty-intensity}
and its corollary, for any $\epsilon>0$ such that $1-\big[\exp(\epsilon)-1\big]^{1/2}>0$
we have $R_{1}^{\epsilon}=R_{1}$ and $\|\lambda-\lambda^{\epsilon}\|_{\infty}\leq-\log\left\{ 1-\big[\exp(\epsilon)-1\big]^{1/2}\right\} $.
Consequently we can apply \cite[Proposition 11, Theorem 21, Corollary 22]{2018arXiv180705421D}
and deduce the first three claims. The fourth claim is direct and
obtained by density of $\mathbf{C}_{0}(E)$ in $L^{2}(\mu)$. For
the fifth claim note that for any $t\geq0$, $\epsilon>0$, $f\in L^{2}(\mu)$
and $\tilde{f}\in\mathbf{C}_{0}(E)$ we have
\begin{align*}
\|P_{t}f-P_{t}^{\epsilon}f\|_{\mu} & \leq\|P_{t}f-P_{t}\tilde{f}\|_{\mu}+\|P_{t}\tilde{f}-P_{t}^{\epsilon}\tilde{f}\|_{\mu}+\|P_{t}^{\epsilon}f-P_{t}^{\epsilon}\tilde{f}\|_{\mu}\\
 & \leq2\|f-\tilde{f}\|_{\mu}-2\log\left\{ 1-\big[\exp(\epsilon)-1\big]^{1/2}\right\} t\|\tilde{f}\|_{\infty},
\end{align*}
where we have used the contraction property of $P_{t}$ and $P_{t}^{\epsilon}$,
the first claim and the corollary of Proposition \ref{prop:bound-approx-penalty-intensity}.
Now for any $\varepsilon>0$, by density of $\mathbf{C}_{0}(E)$ in
$L^{2}(\mu)$ we can find $\tilde{f}\in\mathbf{C}_{0}(E)$ such that
$2\|f-\tilde{f}\|_{\mu}\leq\varepsilon/2$ and $\epsilon_{0}>0$ such
that for any $0\leq\epsilon\leq\epsilon_{0}$, $-2\log\left\{ 1-\big[\exp(\epsilon)-1\big]^{1/2}\right\} t\|\tilde{f}\|_{\infty}\leq\varepsilon/2$.
We conclude.
\end{proof}

\section{Regularized intensities}
\begin{prop}
\label{prop:bound-approx-penalty-intensity}With the notation of Proposition
\ref{prop:example-regularized-intensities} we have the alternative
expression for $r\geq0$ and $\epsilon\geq0$,
\[
\phi_{\epsilon}(r)=\int\min\big\{1,r\exp(-\epsilon/2+\epsilon^{1/2}z)\big\}\mathcal{N}(z;0,1){\rm d}z\;,
\]
$\phi_{\epsilon}(r)<1$ for $\epsilon>0$ and
\[
0\leq\phi_{0}(r)-\phi_{\epsilon}(r)\leq\phi_{0}(r)\big[\exp(\epsilon)-1\big]^{1/2}.
\]
\end{prop}

\begin{proof}
First claim. This is direct for $r=0$. For $\epsilon,r>0$ , with
$A_{r}:=\big\{ z\in\mathbb{R}\colon\log r-\epsilon/2+\epsilon^{1/2}z\geq0\big\}$
\begin{align*}
\phi_{\epsilon}(r): & =1-\Phi\big(\epsilon^{1/2}/2-\epsilon^{-1/2}\log r\big)+r\int_{A_{r}^{c}}\exp\big(-\epsilon/2+\epsilon^{1/2}z\big)\mathcal{N}(z;0,1){\rm d}z.
\end{align*}
Noting that $\epsilon^{1/2}z-z^{2}/2=-(z-\epsilon^{1/2})^{2}/2+\epsilon/2,$
and together with $Z\sim\mathcal{N}(0,1)$, we have
\begin{align*}
(2\pi)^{-1/2}\int_{A_{r}^{c}}\exp\Bigl(-\epsilon/2+\epsilon^{1/2}z-\frac{1}{2}z^{2}\Bigr){\rm d}z & =\mathbb{P}\left(Z+\epsilon^{1/2}<\epsilon^{1/2}/2-\epsilon^{-1/2}\log r\right)\\
 & =\mathbb{P}\left(Z<-\epsilon^{1/2}/2-\epsilon^{-1/2}\log r\right)\\
 & =1-\Phi\big(\epsilon^{1/2}/2+\epsilon^{-1/2}\log r\big),
\end{align*}
and we conclude. Second claim. The leftmost inequality follows from
Jensen's inequality (twice). Since for $a,b>0$, $\min\big\{1,ab\big\}\geq\min\big\{1,a\big\}\min\big\{1,b\big\}$
and using the expressions for the mean and variance of the log-normal
distribution,
\begin{align*}
\phi_{0}(r)-\phi_{\epsilon}(r) & =\int\big[\min\big\{1,r\big\}-\min\big\{1,r\exp\big(-\epsilon/2+\epsilon^{1/2}z\big)\big\}\big]\mathcal{N}(z;0,1){\rm d}z\\
 & \leq\min\big\{1,r\big\}\int\big[1-\min\big\{1,\exp\big(-\epsilon/2+\epsilon^{1/2}z\big)\big\}\big]\mathcal{N}(z;0,1){\rm d}z\\
 & =\phi_{0}(r)\int\max\big\{0,1-\exp\big(-\epsilon/2+\epsilon^{1/2}z\big)\big\}\mathcal{N}(z;0,1){\rm d}z\\
 & \leq\phi_{0}(r)\int\vert1-\exp\big(-\epsilon/2+\epsilon^{1/2}z\big)\vert\mathcal{N}(z;0,1){\rm d}z\\
 & \leq\phi_{0}(r)\big[\exp(\epsilon)-1\big]^{1/2}.
\end{align*}
\end{proof}
\begin{cor}
As a result we have for $r>0$
\[
0\leq1-\frac{\phi_{\epsilon}(r)}{\phi_{0}(r)}\leq\big[\exp(\epsilon)-1\big]^{1/2}
\]
and consequently for $\epsilon>0$ such that $1-\big[\exp(\epsilon)-1\big]^{1/2}>0$
and $\lambda^{\epsilon}$ as in Proposition \ref{prop:example-regularized-intensities},
for any $(x,v)\in E$,
\begin{align*}
0 & \leq\lambda^{\epsilon}(x,v)-\lambda^{0}(x,v)=\log\left\{ \frac{\phi_{0}\big[\exp(\partial U(x)v)\big]}{\phi_{\epsilon}\big[\exp(\partial U(x)v)\big]}\right\} \\
 & \leq-\log\left\{ 1-\big[\exp(\epsilon)-1\big]^{1/2}\right\} .
\end{align*}
\end{cor}

\bibliographystyle{plain}
\bibliography{mdb}

\end{document}